\RequirePackage{amsthm}
\documentclass[sn-basic, Numbered]{sn-jnl}
\usepackage{amsmath,amssymb,amsfonts}%
\usepackage{amsthm}%
\usepackage{mathrsfs}%
\usepackage[title]{appendix}%
\usepackage{xcolor}%
\usepackage{textcomp}%
\usepackage{manyfoot}%
\usepackage{booktabs}%
\usepackage[T1]{fontenc}
\usepackage{orcidlink}
\usepackage{enumitem}
\usepackage{bbm}

\newcommand{\n}{\mathbb{N}}
\newcommand {\<}{\left\langle}  
\renewcommand {\>}{\right\rangle}  
\newcommand {\norma}[1]{\left\|#1\right\|}
\newcommand{\ew}{\mathbb{E}}
\newcommand{\pr}{\mathbb{P}}
\newcommand{\wlim}{w^*\mbox{-}\lim}
\newcommand{\wcl}{w^*\mbox{-}\operatorname{cl}}
\newcommand{\olambda}{\bar{\lambda}}
\newcommand{\ulambda}{\underline{\lambda}}
\newcommand{\hlambda}{\hat{\lambda}}
\newcommand{\tG}{\widetilde{G}}
\newcommand{\tW}{\widetilde{W}}
\newcommand{\dpsi}{D_{\psi}}

\theoremstyle{thmstyleone}%
\newtheorem{theorem}{Theorem}[section]
\newtheorem{proposition}[theorem]{Proposition}%
\newtheorem{lemma}[theorem]{Lemma}
\newtheorem{corollary}[theorem]{Corollary}

\theoremstyle{thmstyletwo}%
\newtheorem{remark}[theorem]{Remark}%

\theoremstyle{thmstylethree}%
\newtheorem{assumption}[theorem]{Assumption}

\numberwithin{equation}{section}

\begin{document}

\title[On the Existence and Uniqueness of Stationary Distributions for Some PDMPs]{On the Existence and Uniqueness of Stationary Distributions for Some Piecewise Deterministic Markov Processes with State-Dependent Jump Intensity}


\author*{\fnm{Dawid} \sur{Czapla} \orcidlink{0000-0002-0562-773X}
}\email{dawid.czapla@us.edu.pl}

\affil{\orgdiv{Institute of Mathematics}, \orgname{University of Silesia in Katowice}, \orgaddress{\street{Bankowa~14}, \city{Katowice}, \postcode{40-007}, \country{Poland}}}


\abstract{In this paper, we consider a subclass of piecewise deterministic Markov processes with a~Polish state space that involve a deterministic motion punctuated by random jumps, occurring in a Poisson-like fashion with some state-dependent rate, between which the trajectory is driven by one of the given semiflows. We prove that there is a one-to-one correspondence between stationary distributions of such processes and those of the Markov chains given by their post-jump locations. Using this result, we further establish a criterion guaranteeing the existence and uniqueness of the stationary distribution in a particular case, where the post-jump locations result from the action of a random iterated function system with an arbitrary set of transformations.}




\maketitle


\section{Introduction}
Let $Y$ be a Polish metric space, and suppose that we are given a finite collection \hbox{$\{S_i:\, i\in I\}$} of continuous semiflows acting from $\mathbb{R}_+ \times Y$ to~$Y$, where $\mathbb{R}_+:=[0,\infty)$. The object of our study will be a~piecewise deterministic Markov process (PDMP) $\Psi:=\{(Y(t),\xi(t))\}_{t\in\mathbb{R}_+}$, evolving on the space $X:=Y\times I$ in such a way that \vspace{-0.2cm}
\begin{equation}
\label{def:pdmp}
Y(t)=S_{\xi_n}(t-\tau_n,Y_n),\quad \xi(t)=\xi_n\quad\text{whenever}\quad t\in [\tau_n,\tau_{n+1})\;\;\text{for}\;\;n\in\n_0.\vspace*{-0.1cm}
\end{equation}
Here, $\Phi:=\{(Y_n,\xi_n)\}_{n\in\n_0}$ stands for a given Markov chain describing the post-jump locations of $\Psi$, and $\{\tau_n\}_{n\in\n_0}$ is an almost surely (a.s.) increasing to infinity sequence of non-negative random variables, representing the jump times, such that the inter-jump times  \hbox{$\Delta \tau_{n+1}:=\tau_{n+1}-\tau_{n}$}, $n\in\n_0$, are independent and satisfy\vspace{-0.1cm}
$$\pr(\Delta\tau_{n+1}\leq t\,|\,\Phi_n=(y,i))=1-\exp\left(-\int_0^t \lambda(S_i(h,y))dh\right)\;\;\text{for all}\;\; t\in\mathbb{R}_+,\; (y,i)\in X,\vspace{-0.1cm}$$ 
with a given bounded continuous function $\lambda:Y\to (0,\infty)$, for which $\inf_{y\in Y}\lambda(y)>0$. The~transition law of $\Phi$, further denoted by $P$, will be defined using an arbitrary stochastic kernel $J$ on $Y$ and a~state-dependent stochastic matrix $\{\pi_{ij}:\,i,j\in I\}$, consisting of continuous functions from~$Y$ to~$[0,1]$, so that $J(y,B)$ is the probability of entering a Borel set $B\subset Y$ by~$Y_n$ given~$Y(\tau_n-)=y$, while $\pi_{ij}(y)$ specifies the probability of transition from $\xi_{n-1}=i$ to $\xi_n=j$ given $Y_n=y$.

The main goal of the paper is to prove that there is a one-to-one correspondence between the families of stationary (invariant) distributions of the process $\Psi$ and the chain $\Phi$, assuming that the kernel $J$ enjoys a strengthened form of the Feller property (Assumption \ref{cnd:f}). This result \hbox{(i.e., Theorem~\ref{thm:main})} generalizes \hbox{\cite[Theorem 5.1]{b:czapla_exp_erg}}, which has been established in the case where $\Delta\tau_n$ are exponentially distributed with a constant rate~$\lambda$, and $\pi_{ij}$ are constant. The latter, in turn, was intended to extend the scope of \hbox{\cite[Theorem 4.4]{b:czapla_erg}}, which refers to a model with a specific form of the kernel $J$.

Originally, PDMPs were introduced by Davis in \cite{b:davis} (see also \cite{b:davis_book}) as a general class of non-diffusive stochastic processes that combine a deterministic motion and random jumps. In this traditional framework, such processes evolve on the union of a countable indexed family of open sets in Euclidean spaces. Associated with each component of the state space is a flow generated by a locally Lipschitz continuous vector field, determining the trajectory's evolution within this component between the jumps. The process $\Psi$, investigated here, is defined in a similar manner, but it evolves on a~general Polish space, and the semiflows are not generated by vector fields but given \hbox{\emph{a priori}}. At~the same time, let us emphasize that our model does not involve the so-called active boundaries (forcing the jumps), which occur in  \cite{b:davis, b:davis_book}. Nevertheless, everything indicates that the techniques employed in this paper should also prove effective in their presence, and it is highly likely that we will present such results in another article.

For the classical PDMPs, analogous results to what we aim to achieve here were established in \hbox{\cite[Theorems 1--3]{b:costa0}} and \hbox{\cite[Ch.~34]{b:davis_book}}. Their proofs are based on the concept of extended generator of a Markov semigroup (see \hbox{\cite[Definition 14.15]{b:davis_book}}), defined as an \hbox{extension} of its strong generator (see \cite[Ch. 1]{b:ethier}) using a martingale \hbox{interpretation} of the~Dynkin identity. Specifically, the key observation (resulting from \hbox{\cite[Proposition 34.7]{b:davis_book}}) is as follows: If, for some finite Borel measure~$\mu$, the extended (or strong) generator $\mathfrak{A}$ of a~given Markov semigroup satisfies
\begin{equation}
\label{e:generator}
\int  \mathfrak{A} f\,d\mu=0\quad\text{for all}\;\; f\in D(\mathfrak{A}),
\end{equation}
with $D(\mathfrak{A})$ standing for the domain of $\mathfrak{A}$, then $\mu$ is invariant for this semigroup, \hbox{provided} that $D(\mathfrak{A})$ intersected with the space of bounded Borel measurable functions separates probability measures. On the other hand, Davis (in \hbox{\cite[Theorem 5.5]{b:davis}}; cf. also \hbox{\cite[Theorem~26.14]{b:davis_book}}) has completely characterized the extended generators of classical PDMPs and shown that their domains enjoy the separating property (see \hbox{\cite[Proposition 34.11]{b:davis_book}}), which enabled the use of the aforementioned fact in the reasoning presented in \cite{b:davis_book, b:costa0}. In our framework, establishing such a property would be rather challenging, if not impossible, as the state space is not Euclidean, and the deterministic dynamics do not correspond to systems of differential equations, which is essential for the discussion in \hbox{\cite[\S 31--32]{b:davis_book}}, leading to \hbox{\cite[Proposition 34.11]{b:davis_book}}. 

A result similar to those discussed above, but in a much more simple setting, where the state space is compact, and randomness of the examined process arises solely from the switching between semiflows, which occur with a constant intensity, was established in \cite[Proposition 2.4]{b:benaim1}. When proving this proposition, the authors utilize the compactness of the state space to show that the Markov semigroup under consideration (which enjoys the Feller property) is strongly continuous on the space of (bounded) continuous functions. According to the Hille--Yosida theorem (\hbox{\cite[Theorem 1.2.6]{b:ethier}}), this guarantees that the domain of its strong generator is dense in such a space and, therefore, separates the measures. Obviously, in our framework, this argument fails due to the lack of compactness.

For the purposes of our study, we employ the concept of weak generator (\hbox{\cite{b:dynkin_1965, b:dynkin_2000}}), similarly to the approach in \cite{b:czapla_erg}. This choice enable us to use an argument resembling the one utilized in \cite{b:benaim1}, but in the context of the \emph{weak star} ($w^*$-)topology. More precisely, by showing that the transition semigroup of $\Psi$ is Feller (which follows from the corresponding property of $J$) and continuous in the weak sense on the space of bounded continuous functions, one can conclude that the \hbox{$w^*$-closure} of the domain of its weak generator contains that space. This, in turn, allows one to argue the \hbox{$\Psi$-invariance of measures} by verifying \eqref{e:generator} for such a generator. Compared to the results of \cite{b:czapla_exp_erg, b:czapla_erg}, the primary contribution of the present paper lies in Lemmas \ref{lem:main1} and~\ref{lem:tech2}, which allow overcoming the difficulty arising from the fact that the jump intensity $\lambda$ is not constant. In particular, the latter enables us to establish (in Lemma~\ref{lem:Pt_prop}) a suitable generalization of \hbox{\cite[Lemma 5.1]{b:czapla_exp_erg}}, which underlies the arguments used to prove \hbox{\cite[Theorem 5.1]{b:czapla_exp_erg}} (itself based on the proof of \hbox{\cite[Theorem 4.4]{b:czapla_erg}}).

We finalize the paper by applying our main result to a special case of the process~$\Psi$, where $J$ is the transition law of the Markov chain arising from a random iterated functions system involving an arbitrary family  $\{w_{\theta}:\,\theta\in\Theta\}$ of continuous transformations from $Y$ into itself (see, e.g., \cite{b:kapica}). Specifically, we provide a set of user-friendly conditions on the component functions guaranteeing that $\Psi$ has a unique stationary distribution with \hbox{finite} first moment (see Theorem \ref{thm:app}), which leads to a generalization of \hbox{\cite[Corollary 4.5]{b:czapla_erg}} (cf.~also \hbox{\cite[Theorem 5.3.1]{b:horbacz}}). In an upcoming paper,  under similar assumptions, we also plan to prove the exponential ergodicity of~$\Psi$ in the bounded Lipschitz distance (and thus generalize~\hbox{\cite[Proposition 7.2]{b:czapla_exp_erg}}). 

Eventually, it is worth to emphasize that considering the jumps occurring with a~state-dependent intensity is often significant in applications. For example, the above-mentioned special case of $\Psi$ with $E=\mathbb{R}_+$, $\Theta$ being an compact interval, and \hbox{$w_{\theta}(y)=y+\theta$} proves to be useful in analysing the stochastic dynamics of gene expression in the presence of transcriptional bursting (see, e.g., \cite{b:tyran} or~\hbox{\cite[\S 5.1]{b:czapla_erg}}). In short, $\{Y(t)\}_{t\in\mathbb{R}_+}$ then describes the concentration of a protein encoded by some gene of a prokaryotic cell. The protein molecules undergo a degradation process, which is interrupted by production appearing in the so-called bursts at random times~$\tau_n$. From a~biological point of view, it is known that the intensity of these bursts depends on the current number of molecules, and thus taking into account the non-constancy of~$\lambda$ makes the model more accurate. 

The outline of the paper is as follows. In Section \ref{sec:prel}, we introduce notation and review several basic concepts related to Markov processes and weak generators of contraction semigroups. Section~\ref{sec:def} presents a formal construction of the model under consideration. The~main result is formulated in Section \ref{sec:main}. Here, we also provide a simple observation regarding the finiteness of the first moments of the invariant measures under consideration. The proof of the main theorem, along with the statements of all necessary auxiliary facts, is given in Section~\ref{sec:proof}. The latter are proved in Section~\ref{sec:proof_aux}. Finally, Section~\ref{sec:application} discusses an application of the main result to the aforementioned specific case of $\Psi$.

\section{Preliminaries} \label{sec:prel}
Consider an arbitrary metric space $E$, and let $\mathcal{B}(E)$ stand for its Borel $\sigma$-field. By $B_b(E)$ we will denote the space of all bounded Borel measurable functions from $E$ to~$\mathbb{R}$, endowed with the supremum norm, i.e., $\norma{f}_{\infty}:=\sup_{x\in E} |f(x)|$ for $f\in B_b(E)$, whilst by $C_b(E)$ we will mean the subspace of $B_b(E)$ consisting of all continuous functions. Further, let $\mathcal{M}_{sig}(E)$ be the family of all finite signed Borel measures on~$E$ (that is, all  \hbox{$\sigma$-additive} real-valued set functions on $\mathcal{B}(E)$), and let $\mathcal{M}(E)$, $\mathcal{M}_1(E)$ stand for its subsets containing all non-negative measures and all probability measures, respectively. Additionally, for any given Borel measurable function $V:E\to [0,\infty)$, the symbol $\mathcal{M}_1^V(E)$ will denote the set of all $\mu\in\mathcal{M}_1(E)$ with finite first moment w.r.t.~$V$, i.e., such that $\int_E V\,d\mu<\infty$. Moreover, for notational brevity, given $f\in B_b(E)$ and $\mu\in\mathcal{M}(E)$, we will often write $\<f,\mu\>$ for the Lebesgue integral $\int_E f \,d\mu$.

\subsection{Markov processes and their transition laws}\label{sec:markov}
Let us now recall several basic concepts from the theory of Markov processes, which we will refer to throughout the paper.

A function $K: E\times\mathcal{B}(E)\to [0,\infty]$ is said to be a \emph{(transition) kernel} on $E$ whenever, for every $A\in\mathcal{B}(E)$, the map \hbox{$E\ni x \mapsto K(x,A)$} is Borel measurable, and, for every \hbox{$x\in E$},  $\mathcal{B}(E)\ni A \mapsto K(x,A)$ is a non-negative Borel measure. If, additionally, $\sup_{x\in E} K(x,E)<\infty$, then $K$ is said to be \emph{bounded}. In the case where $K(x,E)=1$ for every $x\in E$, $K$ is called a \emph{stochastic} kernel (or a \emph{Markov} kernel) and usually denoted by $P$ rather than $K$.

For any two kernels $K_1$ and $K_2$, we can consider their \emph{composition} $K_1K_2$ of the form
$$
K_1K_2(x,A):=\int_E K_2(y,A)K_1(x,dy)\quad\text{for}\quad x\in E,\; A\in\mathcal{B}(E).$$
The iterates of a kernel $K$ are defined as usual by $K^1:=K$ and $K^{n+1}:=K K^n$ for every $n\in\n$. Obviously, the composition of any two bounded kernels is again bounded.

Given a kernel $K$ on $E$, for any non-negative Borel measure $\mu$ on $E$ and any bounded below Borel measurable function $f:E\to\mathbb{R}$, we can consider the measure $\mu K$ and the function $Kf$ defined as
\begin{gather}
\label{e:mK} \mu K(A):=\int_E K(x,A)\mu(dx)\;\;\text{for}\;\;A\in\mathcal{B}(E),\\
\label{e:Kf} Kf(x):=\int_E f(y)\,K(x,dy)\;\;\text{for}\;\;x\in E,
\end{gather}
respectively. They are related to each other in such a way that $\<f, \mu K\>=\<Kf,\mu\>$. Obviously, if $K$ is bounded, then the operator $\mu \mapsto \mu K$ transforms $\mathcal{M}(E)$ into itself, whilst $f\mapsto Kf$ maps $B_b(E)$ into itself. In the case where $K$ is stochastic, the operator defined by~\eqref{e:mK} also leaves the set $\mathcal{M}_1(X)$ invariant.

Let us stress here that the notation consistent with \eqref{e:mK} and \eqref{e:Kf} will be used for all the kernels considered in the paper, without further emphasis.

A kernel $K$ on $E$ is said to be \emph{Feller} if $Kf\in C_b(E)$ for every function $f\in C_b(E)$. Furthermore, a~non-negative Borel measure $\mu$ is called \emph{invariant} for a kernel~$K$ (or for the operator on measures induced by this kernel) whenever $\mu K=\mu$. These two concepts can also be used in reference to any family of transition kernels. The family of this kind is said to be Feller if all its members are Feller. A measure $\mu$ is called invariant for such a family if~$\mu$ is invariant for each of its members.

For a given $E$-valued time-homogeneous Markov chain $\Phi=\{\Phi_n\}_{n\in\n_0}$, defined on a probability space $(\Omega,\mathcal{F},\pr)$, a stochastic kernel $P$ on $E$ is called the \emph{transition law} of this chain~if
$$\pr(\Phi_{n+1}\in A\,|\,\Phi_n=x)=P(x,A)\quad\text{for all}\quad x\in E,\; A\in\mathcal{B}(E),\;n\in\n_0.$$
Letting $\mu_n$ denote the distribution of $\Phi_n$ for each $n\in\n_0$, we then have $\mu_{n+1}=\mu_n P$ for every \hbox{$n\in\n_0$}. The operator $(\cdot)P$ on $\mathcal{M}_1(E)$ is thus referred to as the \emph{transition operator} of~$\Phi$, and  the invariant probability measures of $P$ are simply the \emph{stationary distributions} of $\Phi$. Moreover, in fact, 
$\pr(\Phi_{n+k} \in A \,|\, \Phi_k=x)=P^n(x,A)$ for any $x\in E$, $A\in\mathcal{B}(E)$ and $k,n\in\n_0$, which implies that
$$\ew_x[f(\Phi_n)]=P^n f(x) \quad\text{for any}\quad x\in E,\; f\in B_b(E),\;n\in\n_0,$$
where $\ew_x$ stands for the expectation with respect to $\pr_x:=\pr(\cdot\,|\,\Phi_0=x)$.

On the other hand, it is well-known (see, e.g., \cite{b:revuz}) that, for every stochastic kernel $P$ on~a separable metric space $E$, there exists an $E$-valued Markov chain with transition law $P$. More precisely, putting $\Omega:=E^{\mathbb{N}_0}$, \hbox{$\mathcal{F}=\mathcal{B}(E^{\mathbb{N}_0})$} and, for each $n\in\n_0$, defining $\Phi_n:\Omega\to E$~as
\begin{equation}
\label{def:phi_n}
\Phi_n(\omega):=x_n\quad\text{for}\quad \omega=(x_0,x_1,\ldots)\in E,
\end{equation}
one can show that, for every $x\in E$, there exists a probability measure $\mathbb{P}_x$ on $\mathcal{F}$ such that, for any $n\in\n_0$, $A_0,\ldots,A_n\in\mathcal{B}(E)$ and $F:=\{\Phi_0\in A_0,\ldots,\Phi_n\in A_n\}$, we have
\begin{align}
\begin{split}
\label{e:canonical}
\mathbb{P}_x(F)=&\mathbbm{1}_{A_0}(x)\int_{E}\int_E\ldots \int_{E}\mathbbm{1}_{A_1\times\ldots\times A_n}(x_1,\ldots,x_n)P(x_{n-1},dx_n)\ldots\\
&  P(x_1,dx_2)P(x,dx_1).
\end{split}
\end{align}
Then $\{\Phi_n\}_{n\in\n_0}$ specified by \eqref{def:phi_n} is a time-homogeneous Markov chain on $(\Omega,\mathcal{F},\pr_x)$ with initial state $x$ and transition law $P$. The Markov chain constructed in this way is called the \emph{canonical} one.

Finally, recall that a family $\{P_t\}_{t\in\mathbb{R}_+}$ of stochastic kernels on~$E$ (or the corresponding family of operators on $\mathcal{M}_1(X)$ or $B_b(X)$) is called a \emph{Markov transition semigroup} whenever $P_{s+t}=P_sP_t$ for all $s,t\geq 0$ and $P_0(x,\cdot)=\delta_x$ for every $x\in E$, where $\delta_x$ stands for the Dirac measure at~$x$. While using this term in the context of a~time-homogeneous Markov process $\Psi=\{\Psi(t)\}_{t\geq 0}$, defined on some probability space $(\Omega,\mathcal{F},\pr)$, we will mean that
$$\pr(\Psi(s+t)\in A\,|\, \Psi(s)=x)=P_t(x,A) \quad\text{for all}\quad x\in E,\; A\in\mathcal{B}(E),\;s,t\geq 0.$$
Clearly, if $\mu_t$ denotes the distribution of $\Psi(t)$ for every $t\geq 0$, then $\mu_{s+t}=\mu_s P_t$ for any $s,t\geq 0$, which, in particular, shows that the invariant probability measures of $\{P_t\}_{t\in\mathbb{R}_+}$ are, in fact, the stationary distributions of the process $\Psi$. Moreover, analogously to the discrete case, we have
\begin{equation}
\label{e:efp}
\ew_x[f(\Psi(t))]=P_t f(x) \quad\text{for any}\quad x\in E,\; f\in B_b(E),\;t\geq 0,
\end{equation}
where $\ew_x$ stands for the expectation with respect to $\pr_x:=\pr(\cdot\,|\,\Psi(0)=x)$.

\subsection{Weak infinitesimal generators}\label{sec:generators}

As mentioned in the introduction, we shall use certain tools relying on the concept of a weak infinitesimal generator, which generally pertains to contraction semigroups on subspaces of Banach spaces. In our study, we adopt this concept from \cite{ b:dynkin_1965, b:dynkin_2000}.

Let us first recall that, given a normed space $(L, \norma{\cdot}_L)$, a~family $\{H_t\}_{t\in\mathbb{R}_+}$ of linear operators from $L$ to itself is called a \emph{contraction semigroup on $L$} if $H_0=\operatorname{id}_L$, $H_{s+t}=H_s\circ H_t$ for any $s,t\geq 0$,  and $\norma{H_t f}_L\leq \norma{f}_L$ for all $t\geq 0$ and $f\in L$. 

In what follows, we will only focus on the case where $L$ is a subspace of $(B_b(E), \norma{\cdot}_{\infty})$. Note that any Markov transition semigroup, regarded as a family of operators on~$B_b(E)$ is a contraction semigroup of linear operators. Obviously, if such a semigroup enjoys the Feller property, then it forms a semigroup on $C_b(E)$.

To introduce the notion of the weak generator (adapted to our purposes), let us first consider the Banach space $(\mathcal{M}_{sig}(E), \norma{\cdot}_{TV})$ with the total variation norm, defined by
$$\norma{\mu}_{TV}:=\sup\{|\<f,\mu\>|:\,f\in B_b(E),\;\norma{f}_{\infty}\leq 1\}\quad\text{for}\quad\mu\in\mathcal{M}_{sig}(E),$$
and let $\mathcal{M}_{sig}(E)^*$ denote its dual space. Further, for every $f\in B_b(E)$, define the bounded linear functional $\ell_f:\mathcal{M}_{sig}(E)\to\mathbb{R}$ by
$$\ell_f(\mu):=\<f,\mu\>\quad\text{for}\quad\mu\in\mathcal{M}_{sig}(E).$$
Then $f\mapsto \ell_f$ is an isometric embedding of $(B_b(E), \norma{\cdot}_{\infty})$ in $\mathcal{M}_{sig}(E)^*$ (see \hbox{\cite[p. 50]{b:dynkin_1965}}), i.e., an injective linear map satisfying $\norma{\ell_f}=\norma{f}_{\infty}$ for all $f\in B_b(E)$, where $\norma{\cdot}$ denotes the operator norm. Therefore,~$B_b(E)$ can be regarded as a subspace of $\mathcal{M}_{sig}(E)^*$, and thus it can be endowed with the weak $w^*$-topology inherited from the latter. Moreover, it is easy to check that $B_b(E)$ is \hbox{$w^*$-closed} in~$\mathcal{M}_{sig}(E)^*$. 

In view of the above, a sequence $\{f_n\}_{n\in\n}\subset B_b(E)$ is said to be weak$^*$ convergent to~\hbox{$f\in B_b(E)$}, which is written as $\wlim_{n\to \infty} f_n=f$, whenever $\{\ell_{f_n}\}_{n\in\n}$ converges weakly$^*$ to $\ell_f$ in $\mathcal{M}_{sig}(E)^*$, i.e.,

\begin{equation} \label{wlimit} \wlim_{n\to \infty} f_n=f \;\;\;\mbox{iff}\;\;\;\lim_{n\to\infty}\<f_n,\mu\>=\<f,\mu\>\;\;\;\mbox{for all}\;\;\;\mu\in\mathcal{M}_{sig}(E).
\end{equation}
On the other hand, it is well-known that \eqref{wlimit} is equivalent to the pointwise convergence of $\{f_n\}_{n\in\n}$ to $f$ in conjunction with the boundedness of the sequence $
\{\norma{f_n}_{\infty} \}_{n\in\n}$.

Given a subspace $L$ of $B_b(E)\subset \mathcal{M}_{sig}(E)^*$ and a contraction semigroup $\{H_t\}_{t\in\mathbb{R}_+}$ of linear operators on $L$, by \emph{the~weak (infinitesimal) generator} of this semigroup we will mean (following \cite[Ch.1\,\S\,6]{b:dynkin_1965}) the operator $A_H:D(A_H)\to L_{0,H}$ given by
$$A_H f=\wlim_{t\to 0}\frac{H_t f -f}{t}\quad\text{for}\quad f\in D(A_H),$$
where
$$D(A_H):=\left\{f\in L:\, \wlim_{t\to 0}\frac{H_t f -f}{t}\;\;\text{exists and belongs to}\;\;L_{0,H}\right\},$$
while $L_{0,H}$ denotes the center of $\{H_t\}_{t\geq 0}$, i.e.
\begin{equation}\label{defL0} 
L_{0,H}:=\{f\in L:\,\wlim_{t\to 0} H_t f =f\}.
\end{equation}

At the end of this section, let us quote several basic properties of weak generators that will be useful in the further course of the paper.
\begin{remark}[see \text{\cite[p. 40]{b:dynkin_1965} or \cite[pp. 437-448]{b:dynkin_2000}}]\label{rem:inf}
Let $\{H_t\}_{t\geq 0}$ be a contraction semigroup of linear operators on a subspace $L$ of $B_b(E)$, and let $A_H$ stand for the weak generator of this semigroup. Then
\begin{enumerate}[label=(\roman*), leftmargin=*, widest=ii]
\item\label{rem:i}$\wcl D(A_H) =\wcl L_{0,H},$ where $\wcl$ denotes the weak-$*$ closure in $B_b(E)$.
\item\label{rem:ii} For every $f\in D(A_H)$ the derivative 
$$\mathbb{R}_+\ni t\mapsto \frac{d^+ H_t f}{dt}:=\wlim_{h\to 0^+} \frac{H_{t+h} f - H_t f}{h}$$
exists (in $L_{0,H}$) and is $*$-weak continuous from the right. Consequently, 
$$H_t f\in D(A_H)\quad\text{and}\quad\frac{d_+ H_t f}{dt}=A_H H_t f =H_t A_H f\quad\text{for all}\ \quad t\geq 0.$$
Moreover, the Dynkin formula holds, i.e.,
$$H_t f-f=\int_0^t H_s A_Hf\,ds\quad\text{for all}\ \quad t\geq 0.$$
\end{enumerate}
\end{remark}
 

\section{Definition of the model}\label{sec:def}
Let $(Y,\rho_Y)$ be a non-empty complete separable metric space, and let $I$ stand for an arbitrary non-empty finite set endowed with the discrete topology. Moreover, let us introduce
$$X:=Y\times I \quad \text{and} \quad \bar{X}:=X\times\mathbb{R}_+,$$ 
both equipped with the product topologies, upon assuming that $\mathbb{R}_+:=[0,\infty)$ is supplied with the Euclidean topology.

Consider a family $\{S_i:\,i\in I\}$ of jointly continuous semiflows acting from $\mathbb{R}_+\times Y$ to~$Y$. By calling $S_i$ a semiflow we mean, as usual, that
$$S_i(s,S_i(t,y))=S_i(s+t,y)\quad\text{and}\quad S_i(0,y)=y\quad\text{for any}\quad s,t\in\mathbb{R}_+,\;y\in Y.$$
Further, let $\{\pi_{ij}:\, i,j\in I\}$ be a collection of continuous maps from $Y$ to $[0,1]$ such that 
$$\sum_{j\in I} \pi_{ij}(y)=1\quad\text{for all}\quad i\in I,\; y\in Y.$$
Moreover, let $\lambda:Y\to (0,\infty)$ be a continuous function satisfying
\begin{equation}
\label{e:lambda_bounds}
\ulambda \leq \lambda(y) \leq \olambda \quad \text{for every}\quad y\in Y,
\end{equation}
with certain constants $\ulambda, \olambda>0$, and put
$$\Lambda(y,i,t):=\int_0^t \lambda(S_i(h,y))\,dh\quad \text{for}\quad y\in Y,\; i\in I,\; t\in\mathbb{R}_+.$$
Finally, suppose that we are given an arbitrary stochastic kernel $J: Y\times \mathcal{B}(Y)\to [0,1]$.

Let us now define a stochastic kernel $\bar{P}: \bar{X}\times\mathcal{B}(\bar{X})\to [0,1]$ by setting
\begin{align}
\begin{split}
\label{def:P_bar}
\bar{P}((y,i,s), \bar{A})&= \int_0^{\infty} \lambda(S_i(t,y)) e^{-\Lambda(y,i,t)}\\ &\quad\times\int_Y \sum_{j\in I} \pi_{ij}(u)\mathbbm{1}_{\bar{A}}(u,j,s+t)\,J(S_i(t,y),du)\,dt
\end{split}
\end{align}
for any $y\in Y$, $i\in I$, $s\in\mathbb{R}_+$ and $\bar{A}\in\mathcal{B}(\bar{X})$. Furthermore, let $P:X\times\mathcal{B}(X)\to [0,1]$ be given by
\begin{equation}
\label{def:P}
P((y,i),A):=\bar{P}((y,i,0),A\times\mathbb{R}_+)\quad\text{for}\quad y\in Y,\;i\in I,\; A\in\mathcal{B}(X).
\end{equation}

\begin{remark}\label{rem:feller}
Taking into account the continuity of the maps $X\ni (y,i) \mapsto S_i(t,y)$ for $t\geq 0$, $(y,i)\mapsto \pi_{ij}(y)$ for $j\in I$, and $(y,i)\mapsto \lambda(y)$, it is easy to see that $P$ is Feller whenever so is the kernel $J$.
\end{remark}
By $\bar{\Phi}:=\{(Y_n,\xi_n,\tau_n)\}_{n\in\n_0}$ we will denote a time-homogeneous Markov chain with state space $\bar{X}$ and transition law~$\bar{P}$, wherein $Y_n$, $\xi_n$, $\tau_n$ take values in $Y$, $I$, $\mathbb{R}_+$, respectively. For simplicity of analysis, we shall regard $\bar{\Phi}$ as the canonical Markov chain (starting from some point of $\bar{X}$), constructed on the coordinate space $\Omega:=\bar{X}^{\n_0}$, equipped with the $\sigma$-field $\mathcal{F}:=\mathcal{B}\left(\bar{X}^{\n_0}\right)$ and a suitable probability measure $\pr$ on $\mathcal{F}$. Obviously, $\Phi:=\{(Y_n,\xi_n)\}_{n\in \n_0}$ is then a Markov chain with respect to its own natural filtration, governed by the transition law~$P$, given by \eqref{def:P}. Moreover, for every~$n\in\n_0$, we have
\begin{gather}
\nonumber
\pr(Y_{n+1}\in B\,\,|\,\Phi_n;\;S_{\xi_n}( \Delta\tau_{n+1},Y_n)=y)=J(y,B)\quad\text{for}\quad y\in Y,\; B\in\mathcal{B}(Y),\\
\nonumber
\pr(\xi_{n+1}=j\,\,|\,Y_n;\;\xi_n=i, \,Y_{n+1}=y)=\pi_{ij}(y)\quad\text{for}\quad y\in Y,\;i,j\in I,
\end{gather}
where $\Delta \tau_{n+1}:=\tau_{n+1}-\tau_n$, and
\begin{equation}
\label{e:dist_tau}
\pr(\tau_{n+1}\leq t\,|\,\bar{\Phi}_n=(y,i,s))=\mathbbm{1}_{[s,\infty)}(t) \left(1-e^{-\Lambda(y,i,t-s)}\right)
\end{equation}
for  $t\in\mathbb{R}_+$ and $(y,i,s)\in\bar{X}$.

In particular, \eqref{e:dist_tau} implies that the conditional distributions of $\Delta \tau_{n+1}$  given~$\bar{\Phi}_n$ are of the form
$$\pr(\Delta\tau_{n+1}\leq t\,|\,\bar{\Phi}_n=(y,i,s))=1-e^{-\Lambda(y,i,t)}\quad\text{for}\quad t\in\mathbb{R}_+,\; (y,i,s)\in\bar{X}.$$
This yields that $\Delta \tau_n>0$ a.s. for all $n\in\n$ (so $\{\tau_n\}_{n\in\n_0}$ is a.s. strictly increasing), and,  together with the Markov property of $\bar{\Phi}$, shows that $\Delta\tau_1,\Delta \tau_2,\ldots$ are mutually independent. Further, it follows that, for any $n,r\in\n$,
\begin{align*}
\ew[(\Delta\tau_{n})^r]&=\ew\left[\ew[(\Delta\tau_{n})^r\,|\,\bar{\Phi}_{n-1}]\right]=\int_0^{\infty} t^r \ew\left[\lambda(S_{\xi_{n-1}}(t,\xi_{n-1}))e^{-\Lambda(t,Y_{n-1},\xi_{n-1})}\right]dt,
\end{align*}
whence, in view of \eqref{e:lambda_bounds}, we get 
$$
\ulambda \olambda^{-(r+1)}r!\leq  \ew[(\Delta\tau_{n})^r]\leq\olambda \ulambda^{-(r+1)}r!.$$
Consequently, Kolmogorov's criterion guarantees that $(\Delta\tau_n-\ew\Delta\tau_n)_{n\in\n}$ obeys the strong law of large numbers. Hence, writing

$$\tau_n-\tau_0=n\hspace{-0.05cm}\left(\frac{\sum_{k=1}^n (\Delta\tau_k\hspace{-0.05cm}-\ew\Delta\tau_k)}{n}+\frac{\sum_{k=1}^n\ew\Delta\tau_k}{n}\right)\geq n \hspace{-0.05cm}\left(\frac{\sum_{k=1}^n (\Delta\tau_k-\ew\Delta\tau_k)}{n}+\ulambda \olambda^{-2}\right)$$
for $n\in\n$, we can conclude that $\tau_n \uparrow \infty$ a.s.

The main focus of our study will be the PDMP $\Psi:=\{\Psi(t)\}_{t\in\mathbb{R}_+}$ of the form $$\Psi(t):=(Y(t),\xi(t))\quad\text{for}\quad t\in\mathbb{R}_+,$$ 
defined via interpolation of $\Phi$ according to formula \eqref{def:pdmp}. Clearly, this definition is well-posed since $\tau_n \uparrow \infty$ a.s., and $\Phi$ describes the post-jump locations of the process~$\Psi$, that~is,
$$\Phi_n=(Y_n,\xi_n)=(Y(\tau_n),\xi(\tau_n))=\Psi(\tau_n) \quad\text{for every} \quad n\in\n_0.$$ 
In what follows, the Markov transition semigroup of $\Psi$ will be denoted by $\{P_t\}_{t\geq 0}$.

\section{Main results}\label{sec:main}

In this section, we shall formulate the main result of this paper, concerning a \hbox{one-to-one} correspondence between invariant probability measures of the transition semigroup $\{P_t\}_{t\in\mathbb{R}_+}$ of the process~$\Psi$, determined by \eqref{def:pdmp}, and those of the transition operator~$P$ of the chain~$\Phi$, given by \eqref{def:P}. 

For this aim, let us consider two stochastic kernels $G,W:X\times\mathcal{B}(X)\to [0,1]$ given by
\begin{gather}
\label{def:G}
G(x,A):=\int_0^{\infty} \lambda(S_i(t,y))e^{-\Lambda(y,i,t)}\mathbbm{1}_A(S_i(t,y),i)\,dt,\\
\label{def:W}
W(x,A):= \int_Y \sum_{j\in I} \pi_{ij}(u)\mathbbm{1}_A(u,j)\,J(y,du)
\end{gather}
for all $x=(y,i)\in X$ and $A\in\mathcal{B}(X)$, where $J$ stands for the kernel involved in \eqref{def:P_bar}. Further, define
\begin{equation}
\label{def:hlambda}
\hlambda(y,i):=\lambda(y)\quad \text{for any}\quad y\in Y,\;i\in I,
\end{equation}
and introduce two (generally non-stochastic) kernels $\tG,\tW:X\times\mathcal{B}(X)\to [0,\infty)$ of the form
$$\tG(x,A):=G\left(\mathbbm{1}_A/\hlambda\right)\hspace{-0.1cm}(x)\quad\text{and}\quad \tW(x,A):=\hlambda(x)\,W\mathbbm{1}_A(x) \;\;\text{for all}\;\; x\in X,\;A\in \mathcal{B}(X).$$

\begin{remark}\label{rem:gw}
It is easily seen that $GW=\tG\tW=P$, where $P$ is given by~\eqref{def:P}.
\end{remark}

\begin{remark} According to \eqref{e:lambda_bounds}, for any $x\in X$ and $A\in\mathcal{B}(X)$, we have
$$\olambda^{-1}  \leq  \tG (x,A)\leq \ulambda^{-1}\quad \text{and}\quad  \ulambda \leq  \tW (x,A)\leq \olambda.$$
Consequently, the kernels $\tG$ and $\tW$ are bounded, and thus the sets $\mathcal{M}(X)$ and $B_b(X)$ are invariant under the operators induced by these kernels according to \eqref{e:mK} and \eqref{e:Kf}, respectively. Moreover, note that, if $\mu\in\mathcal{M}(X)$ is a non-zero measure, then the measures $\mu \tG$ and $\mu \tW$ are non-zero as well.
\end{remark}

\begin{remark}\label{rem:feller_gw} The kernels $G$ and $\widetilde{G}$ are Feller. Moreover, if the kernel $J$ is Feller then so are the kernels $W$ and~$\widetilde{W}$.
\end{remark}

Essentially, apart from the conditions imposed on the model components in Section~\ref{sec:def}, the only assumption that we need to make in our main theorem is the following strengthened version of the Feller property for the kernel $J$:

\begin{assumption}\label{cnd:f}
For every $g\in C_b(Y\times\mathbb{R}_+)$, the map \hbox{$Y\times\mathbb{R}_+\ni(y,t)\mapsto J g(\cdot,t)(y)$} is jointly continuous.
\end{assumption}
Although the above assumption might appear somewhat technical, it is crucial to ensure the joint continuity of a certain specific function used to obtain an explicit form of the semigroup~$\{P_t\}_{t\in\mathbb{R}+}$ in the upcoming Lemma \ref{lem:Pt_prop}\ref{lem:ii}. This continuity, in turn, will be needed to prove the subsequent Lemma \ref{lem:main2}, which plays a key role in our approach.

The main result of the paper reads as follows:
\begin{theorem} \label{thm:main}
Suppose that the kernel $J$ satisfies Assumption~\ref{cnd:f}.
\begin{enumerate}[label=\textnormal{(\alph*)}, leftmargin=*]
\item\label{cnd:a}If $\mu_*^{\Phi}\in\mathcal{M}_1(X)$ is invariant for $P$, then 
\begin{equation}
\label{e:m_psi}
\mu_*^{\Psi}:=\frac{\mu_*^{\Phi}\tG}{\mu_*^{\Phi}\tG(X)}
\end{equation}
is an invariant probability measure of $\{P_t\}_{t\in\mathbb{R}_+}$, and 
\begin{equation}
\label{e:m_phi}
\mu_*^{\Phi}=\frac{ \mu_*^{\Psi} \tW}{\mu_*^{\Psi}\tW(X)}.
\end{equation}
\item\label{cnd:b} If $\mu_*^{\Psi}\in\mathcal{M}_1(X)$ is invariant for $\{P_t\}_{t\in\mathbb{R}_+}$, then $\mu_*^{\Phi}$ defined by \eqref{e:m_phi} is an invariant probability measure of $P$, and $\mu_*^{\Psi}$ can be then expressed as in \eqref{e:m_psi}.
\end{enumerate}
\end{theorem}
\noindent The proof of Theorem \ref{thm:main}, together with all needed auxiliary results, is given in Section~\ref{sec:proof}.

\begin{remark}
Note that 
$$\mu_*^{\Phi}\tG(X)=\int_X \int_0^{\infty} e^{-\Lambda(x,t)}\,dt\,\mu_*^{\Phi}(dx) \quad\text{and}\quad \mu_*^{\Psi}\tW(X)=\int_X \lambda(x)\mu_*^{\Psi}(dx).$$
Hence, in particular, if $\lambda$ is constant, then $\mu_*^{\Psi}=\mu_*^{\Phi}G$ and $\mu_*^{\Phi}= \mu_*^{\Psi} W.$
\end{remark}

\begin{corollary}\label{cor:main}
Suppose that Assumption \ref{cnd:f} holds. Then $\{P_t\}_{t\in\mathbb{R}_+}$ admits a unique invariant probability measure if and only if so does $P$.
\end{corollary}

Let us finish this section with a simple observation concerning the property of having finite first moments by the measures featured in Theorem \ref{thm:main}. We are interested in the moments with respect to the function $V:X\to [0,\infty)$ given by
\begin{equation}
\label{def:V}
V(x):=\rho_Y(y,y_*)\quad\text{for}\quad x=(y,i)\in X,
\end{equation}
where $y_*$ is an arbitrarily fixed point of $Y$. To state an appropriate result, let us introduce additionally the following two assumptions, which will also be utilized in Section \ref{sec:application}:

\begin{assumption}\label{cnd:a1}
For some point $y_*\in Y$ we have
$$
\beta(y_*):=\max_{i\in I}\int_0^{\infty} e^{-\ulambda t} \rho_Y(S_i(t,y_*),y_*)\,dt<\infty.
$$
\end{assumption}

\begin{assumption}\label{cnd:a2}
There exist constants $L>0$ and $\alpha<\ulambda$ such that
$$
\rho_Y(S_i(t,u), S_i(t,v))\leq Le^{\alpha t}\rho_Y(u,v)\;\;\;\mbox{for}\;\;\;u,v\in Y,\;i\in I,\;t\geq 0.
$$
\end{assumption}
It should be noted here that, if Assumptions \ref{cnd:a1} and \ref{cnd:a2} are fulfilled, then \hbox{$\beta(y_*)<\infty$} for every $y_*\in Y$.
\vspace{0.2cm}

\begin{proposition}\label{prop:m1v}
Let $V$ be the function given by \eqref{def:V}. Then, the following holds:
\begin{itemize}
\item[(i)] Under Assumptions \ref{cnd:a1} and \ref{cnd:a2}, for any $\mu_*^{\Phi}\in\mathcal{M}_1^V(X)$, the measure $\mu_*^{\Psi}$ defined by \eqref{e:m_psi} also belongs to $\mathcal{M}_1^V(X)$.
\item[(ii)] Suppose that there exist constants $\tilde{a},\tilde{b}\geq 0$ such that, for \hbox{$\widetilde{V}:=\rho_Y(\cdot,y_*)$}, we have
\begin{equation}
\label{e:lapunov_J}
J\widetilde{V}(y)\leq \tilde{a} \widetilde{V}(y)+\tilde{b}\quad\text{for all}\quad y\in Y.
\end{equation}

Then, for any $\mu_*^{\Psi}\in\mathcal{M}_1^V(X)$, the measure $\mu_*^{\Phi}$ defined by \eqref{e:m_phi} also belongs to  $\mathcal{M}_1^V(X)$.
\end{itemize}
\end{proposition}
\begin{proof}
To see (i) it suffices to observe that, for every $(y,i)\in X$,
\begin{align*}
\tG V(y,i)&=\int_0^{\infty} e^{-\Lambda(y,i,t)} \rho_Y(S_i(t,y),y_*)\,dt \\
&\leq \int_0^{\infty} e^{-\ulambda t} \rho_Y(S_i(t,y),S_i(t,y_*))\,dt+\int_0^{\infty} e^{-\ulambda t} \rho_Y(S_i(t,y_*),y_*)\,dt\\
&\leq L\left(\int_0^{\infty}e^{-(\ulambda-\alpha)t}\,dt \right) \rho_Y(y,y_*) + \beta(y_*)=\frac{L}{\ulambda-\alpha} V(y,i)+\beta(y_*).
\end{align*}
Statement (ii) follows from the fact that, for any $(y,i)\in X$,
\begin{align*}
\tW V (y,i)& = \lambda(y,i) \int_Y \sum_{j\in I} \pi_{ij}(u) \widetilde{V}(u) J(y,du) =  \lambda(y,i) J\widetilde{V}(y)\\
&\leq \olambda (\tilde{a} \widetilde{V}(y) +\tilde{b})=\olambda\tilde{a}V(y,i) +\olambda \tilde{b}.
\end{align*}
\end{proof}

\section{Proof of the main theorem}\label{sec:proof}
Before we proceed to prove Theorem \ref{thm:main}, let us first formulate certain auxiliary results \hbox{(specifically, Lemmas \ref{lem:Pt_prop}--\ref{lem:main1})}, whose proofs are given in Section \ref{sec:proof_aux}.

The first result (proved in Section \ref{sec:proof_props}) collects certain properties of the transition semigroup ${\{P_t\}}_{t\in\mathbb{R}+}$, which are essential for establishing the forthcoming Lemma \ref{lem:main2}, concerning the weak generator of this semigroup. It is also worth noting that this result extends the scope of \cite[Lemma 5.1]{b:czapla_erg}, which was previously established only for constant $\lambda$ and $\pi_{ij}$.

\begin{lemma}\label{lem:Pt_prop}
The following statements hold for the transition semigroup $\{P_t\}_{t\in\mathbb{R}_+}$ of the process~$\Psi$, specified by \eqref{def:pdmp}:
\begin{enumerate}[label=(\roman*), leftmargin=*, widest=iii]
\item\label{lem:i} If $J$ is Feller, then $\{P_t\}_{t\in\mathbb{R}_+}$ is Feller too.
\item\label{lem:ii} For every $f\in B_b(X)$, there exist functions $u_f:X\times \mathbb{R}_+\to\mathbb{R}$ and $\psi_f:X\times \dpsi \to\mathbb{R}$, with $\dpsi:=\left\{(t,T)\in\mathbb{R}_+^2: t\leq T\right\}$, such that
\begin{equation}
\label{e:P_T_app}
P_T f(x)=e^{-\Lambda(y,i,T)} f(S_i(T,y),i)+\int_0^T \psi_f((y,i),t,T)\,dt+u_f((y,i),T)
\end{equation}
for any $x=(y,i)\in X$ and $T\in\mathbb{R}_+$, and that the following conditions hold:
\begin{equation}
\label{e:prop_uf}
u_f(\cdot,T)\in B_b(X) \quad\text{for all}\quad T\in\mathbb{R}_+,\quad\lim_{T\to 0} \norma{u_f(\cdot,T)}_{\infty}/{T}=0,
\end{equation}
\begin{equation}
\label{e:psi_prop1}
\psi_f(\cdot,t,T)\in B_b(X),\quad \norma{\psi_f(\cdot,t,T)}_{\infty}\leq\olambda\norma{f}_{\infty}\quad\text{for any}\quad (t,T)\in \dpsi,
\end{equation}
\begin{equation}
\label{e:psi_prop2}
\psi_f(x,0,0)=\hlambda(x) Wf(x)\quad\text{for every}\quad x\in X,
\end{equation}
with $\hlambda$ given by \eqref{def:hlambda}, and, if Assumption \ref{cnd:f} is fulfilled, then \hbox{$\dpsi\ni (t,T)\mapsto \psi_f(x,t,T)$} is jointly continuous for every $x\in X$ whenever $f\in C_b(X)$.

\item\label{lem:iii} $\{P_t\}_{t\in\mathbb{R}_+}$ is stochastically continuous, i.e. 
$$\lim_{T\to 0} P_T f(x)=f(x)\quad\text{for any}\quad x\in X,\; f\in C_b(X).$$
\end{enumerate}
\end{lemma}

Obviously, according to statement  \ref{lem:i} of Lemma \ref{lem:Pt_prop}, if $J$ is Feller (or, in particular, if Assumption \ref{cnd:f} is fulfilled), then $\{P_t\}_{t\in\mathbb{R}_+}$ can be viewed as a contraction semigroup of linear operators on $C_b(X)$, and thus we can consider its weak generator (in the sense specified in Section \ref{sec:generators}). In what follows, this generator will be denoted by $A_P$. Apart from this, we also employ the weak generator $A_Q$ of the semigroup $\{Q_t\}_{t\in\mathbb{R}_+}$ defined by
\begin{equation}
\label{def:Qt}
Q_t f(x):=f(S_i(t,y),i)\quad\text{for}\quad x=(y,i)\in X,\;f\in C_b(X).
\end{equation}
Clearly, $D(A_p)$ and $D(A_Q)$ are then subsets of $C_b(X)$.  Furthermore, having in mind \eqref{defL0}, we see that $L_{0,P}=C_b(X)$ by statement \ref{lem:iii} of Lemma~\ref{lem:Pt_prop}, and also \hbox{$L_{0,Q}=C_b(X)$}, due to the continuity of $S_i(\cdot,y)$, $y\in Y$.

The main idea underlying the proof of our main theorem is that   the generator~$A_P$ can be expressed using $A_Q$ and the operator $W$, determined by \eqref{def:W} (similarly as in \hbox{\cite[Theorem 5.5]{b:davis}}), which is exactly the statement of the lemma below. The proof of this result (given in \hbox{Section}~\ref{sec:proof_lems}) is founded on assertion \ref{lem:ii} of Lemma~\ref{lem:Pt_prop}, which, incidentally, is the reason why we require Assumption \ref{cnd:f} rather than just the Feller property of $J$. It should also be noted that, according to Remark \ref{rem:feller_gw}, the kernel $W$ is Feller under this assumption, which makes the formula below meaningful.
\begin{lemma}\label{lem:main2}
Suppose that Assumption \ref{cnd:f} holds. Then $D(A_P)=D(A_Q)$ and, for every \hbox{$f\in D(A_P)$}, we have
\begin{equation}
\label{e:APQ}
A_P f=A_Q f +\hlambda Wf-\hlambda  f,
\end{equation}
where $\hlambda$ is given by \eqref{def:hlambda}, and $W$ is the operator on $C_b(X)$, induced by the kernel specified in \eqref{def:W}.
\end{lemma}

Another fact playing a significant role in the proof of Theorem \ref{thm:main} is related with the \hbox{invertibility} of the operator $G$ on $C_b(X)$ (cf. Remark \ref{rem:feller_gw}), induced by the kernel given by~\eqref{def:G}. Clearly, if $\lambda$ is constant, then $G/\lambda$ is simply the resolvent of the semigroup $\{Q_t\}_{t\in\mathbb{R}_+}$.  As is well known (see \hbox{\cite[Theorem 1.7]{b:dynkin_1965}}), in this case, one has $G/\lambda\,|_{\,C_b(X)}= (\lambda \operatorname{id}-A_Q)^{-1}$, which is a key observation in \cite{b:czapla_erg}. Since this argument fails in the present framework, we will prove (in Section \ref{sec:proof_lems}) the following:

\begin{lemma}\label{lem:main1}
Let $A_Q/\hlambda$ be the operator defined by
$$\left(A_Q/\hlambda\right)f=\frac{A_Q f}{\hlambda}\quad\text{for}\quad f\in D(A_Q),$$
with $\hlambda$ given by \eqref{def:hlambda}.
Then the operator $\operatorname{id}-A_Q/\hlambda: D(A_Q)\to C_b(X)$ is invertible, and its inverse is the operator on $C_b(X)$ induced by the kernel $G$, specified in~\eqref{def:G}. Equivalently, this means that
\begin{gather}
\label{e:inv1}
Gf\in D(A_Q)\quad\text{and}\quad Gf-\frac{A_Q G f}\hlambda=f\quad\text{for any}\quad f\in C_b(X),\\
\label{e:inv2}
Gf-G\left(\frac{A_Q f}{\hlambda}\right)=f\quad \text{for any}\quad f\in D(A_Q).
\end{gather}
\end{lemma}

Armed with the results above, we are now prepared to prove Theorem \ref{thm:main}. It is worth highlighting here that, although the reasoning below explicitly utilizes only Lemmas \ref{lem:main2} and~\ref{lem:main1}, it fundamentally relies on Remark \ref{rem:inf}, which is valid for $\{P_t\}_{t\in\mathbb{R}_+}$ with $L_{0,P}=C_b(X)$ owing to Lemma \ref{lem:Pt_prop}.

\begin{proof}[Proof of Theorem \ref{thm:main}]
{\bf (a):}\; To prove statement \ref{cnd:a}, suppose that $\mu_*^{\Phi}\in\mathcal{M}_1(X)$ is invariant for $P$, and let $\mu_*^{\Psi}$ be given by \eqref{e:m_psi}. 

We will first show that
\begin{equation}
\label{e:main1}
\<A_P f,\, \mu_*^{\Phi}\tG\>=0\quad\text{for every}\quad f\in D(A_P).
\end{equation}
To this end, let $f\in D(A_P)$ and define $\nu_*:=\mu_*^{\Phi} G$. Taking into account that $GW=P$ (cf.~Remark \ref{rem:gw}), we see that
$$\nu_*WG=(\mu_*^{\Phi} G)WG=(\mu_*^{\Phi} GW)G=(\mu_*^{\Phi} P)G=\mu_*^{\Phi} G=\nu_*,$$
which means that $\nu_*$ is an invariant probability measure of $WG$. Further, from Lemma~\ref{lem:main2} it follows that $f\in D(A_Q)$, and that
\begin{equation}
\label{e:main2}
\frac{A_P f}{\hlambda}=\frac{A_Q f}{\hlambda}+Wf-f.
\end{equation}
Using the $WG$-invariance of $\nu_*$ and identity \eqref{e:inv2}, resulting from Lemma \ref{lem:main1}, we further obtain
\begin{align*}
\<f-\frac{A_Q f}{\hlambda},\,\nu_* \>&= \<f-\frac{A_Q f}{\hlambda},\,\nu_* WG\>= \<G\left(f-\frac{A_Q f}{\hlambda}\right),\,\nu_*W \>\\
&=\<f,\,\nu_* W\>=\<Wf,\,\nu_*\>.
\end{align*}
In view of \eqref{e:main2}, this gives
$$\<\frac{A_P f}{\hlambda},\,\nu_*\>=\< \frac{A_Q f}{\hlambda}+Wf-f ,\, \nu_*\>=0,$$
which finally implies that
\begin{align*}
\<A_P f,\, \mu_*^{\Phi}\tG\>&=\<\tG(A_P f),\, \mu_*^{\Phi}\>=\<G\left(\frac{A_P f}{\hlambda}\right),\, \mu_*^{\Phi} \>=\<\frac{A_P f}{\hlambda},\, \mu_*^{\Phi} G \>\\
&=\<\frac{A_P f}{\hlambda},\, \nu_* \>=0,
\end{align*}
as claimed.

Now, according to Remark \ref{rem:inf}\ref{rem:ii}, for any $f\in D(A_P)$ and $t\geq 0$, we have
$$P_tf\in D(A_P)\quad\text{and}\quad P_t f - f = \int_0^t A_P P_s f\,ds,$$
which, together with \eqref{e:main1}, yields that
\begin{align*}
\<P_t f - f,\, \mu_*^{\Phi} \tG\>=\<\int_0^t A_P P_s f\,ds,\, \mu_*^{\Phi} \tG\>=\int_0^t \<A_P P_s f,\, \mu_*^{\Phi} \tG\> \,ds=0.
\end{align*} 
Since  $C_b(X)\subset w^*-\operatorname{cl}D(A_P)$ due to Remark \ref{rem:inf}\ref{rem:i}, one can easily conclude that, in fact, the above equality holds for all $f\in C_b(X)$. We have therefore shown that
$$\<P_t f,\,\mu_*^{\Phi} \tG\>=\<f,\,\mu_*^{\Phi} \tG\>\quad\text{for all}\quad t\in\mathbb{R}_+,\; f\in C_b(X).$$
This obviously implies that, for any $t\in\mathbb{R}_+$ and $f\in C_b(X)$,
\begin{align*}
\<f, \mu_*^{\Psi} P_t\>=\<P_t f,\, \mu_*^{\Psi}\>= \frac{\<P_t f,\, \mu_*^{\Phi} \tG\>}{\mu_*^{\Phi}\tG(X)}=\frac{\<f,\, \mu_*^{\Phi} \tG\>}{\mu_*^{\Phi}\tG(X)}=\<f,\,\mu_*^{\Psi}\>,
\end{align*}
and thus proves that $\mu_*^{\Psi}$ is invariant for $\{P_t\}_{t\in\mathbb{R}_+}$.

Furthermore, since $\tG\tW=P$ (as emphasized in Remark \ref{rem:gw}) and $\mu_*^{\Phi}$ is invariant for $P$, we see that
\begin{align*}
\mu_*^{\Psi} \tW=\frac{\mu_*^{\Phi} \tG\tW}{\mu_*^{\Phi} \tG(X)}=\frac{\mu_*^{\Phi} P}{\mu_*^{\Phi} \tG(X)}=\frac{\mu_*^{\Phi}}{\mu_*^{\Phi} \tG(X)},
\end{align*}
whence, in particular,
$$\mu_*^{\Psi} \tW (X)=\frac{1}{\mu_*^{\Phi} \tG(X)}.$$
These two observations finally yield that
$$\frac{\mu_*^{\Psi} \tW}{\mu_*^{\Psi} \tW(X)}=\mu_*^{\Phi} \tG(X) \cdot \mu_*^{\Psi} \tW=\mu_*^{\Phi},$$
as claimed in \eqref{e:m_phi}.

{\bf (b):}\; We now proceed to the proof of claim \ref{cnd:b}. For this aim, suppose that $\mu_*^{\Psi}\in\mathcal{M}_1(X)$ is invariant for $\{P_t\}_{t\in\mathbb{R}_+}$ and consider the measure $\mu_*^{\Phi}$ defined by \eqref{e:m_phi}.

Let us begin with showing that
\begin{equation}
\label{e:main3}
\<\tW G f, \mu_*^{\Psi}\>=\<\hlambda f, \mu_*^{\Psi}\>\quad\text{for every} \quad f\in C_b(X).
\end{equation}
To do this, let $f\in C_b(X)$ and define $g:=Gf$. Clearly, lemmas \ref{lem:main2} and \ref{lem:main1} guarantee that $g\in D(A_Q)=D(A_P)$. Taking into account statement \ref{rem:ii} of Remark \ref{rem:inf} and the fact that~$\mu_*^{\Psi}$ is invariant for $\{P_t\}_{t\in\mathbb{R}_+}$, we infer that
\begin{align*}
\<A_P g,\, \mu_*^{\Psi}\>&=\int_0^1 \<A_P g,\, \mu_*^{\Psi}\>\,ds=\int_0^1 \<P_s A_P g,\, \mu_*^{\Psi}\>\,ds=\<\int_0^1 P_s A_P g \,ds,\,\mu_*^{\Psi}\>\\
&=\<P_1 g - g,\, \mu_*^{\Psi}\>=\<g,\, \mu_*^{\Psi}P_1\>-\<g,\,\mu_*^{\Psi}\>=0.
\end{align*}
Hence, referring to Lemma \ref{lem:main2}, we get
$$\<A_Q g + \hlambda Wg -\hlambda g,\,\mu_*^{\Psi}\>=\<A_P g,\, \mu_*^{\Psi}\>=0,$$
which gives
$$\<\tW g, \,\mu_*^{\Psi}\>=\<\hlambda W g, \,\mu_*^{\Psi}\>=\< \hlambda g - A_Q g,\,\mu_*^{\Psi}\>
=\<\hlambda\left(g-\frac{A_Q g}{\hlambda}\right),\, \mu_*^{\Psi}\>.$$
Eventually, having in mind that $g=Gf$ and using \eqref{e:inv1}, following from Lemma~\ref{lem:main1}, we obtain
$$\<\tW Gf, \,\mu_*^{\Psi}\>= \<\hlambda\left(Gf-\frac{A_Q(Gf)}{\hlambda}\right),\, \mu_*^{\Psi}\>=\<\hlambda f,\,\mu_*^{\Psi}\>,$$
which is the desired claim.

If we now fix an arbitrary $f\in C_b(X)$ and apply \eqref{e:main3} with $f/\hlambda$ in place of $f$,  then we obtain
$$\<f,\,\mu_*^{\Psi}\>=\<\tW G\left(\frac{f}{\hlambda}\right),\,\mu_*^{\Psi}\>=\<\tW \tG f,\,\mu_*^{\Psi}\>=\<f,\, \mu_*^{\Psi} \tW \tG\>,$$
which shows that $\mu_*^{\Psi}$ is invariant for $\tW \tG$. From this and the identity $\tG \tW=P$ it follows that
$$(\mu_*^{\Psi} \tW )P=(\mu_*^{\Psi} \tW)\tG \tW=(\mu_*^{\Psi} \tW\tG) \tW=\mu_*^{\Psi} \tW,$$
which, in turn, gives
$$\mu_*^{\Phi} P=\frac{\mu_*^{\Psi} \tW P}{\mu_*^{\Psi}\tW(X)}=\frac{\mu_*^{\Psi} \tW}{\mu_*^{\Psi}\tW(X)}=\mu_*^{\Phi}.$$
Hence $\mu_*^{\Phi}$ is indeed invariant for $P$.

Finally, it remains to show that $\mu_*^{\Psi}$ can be represented according to \eqref{e:m_psi}. To see this, we again use the fact that $\mu_*^{\Psi}$ is invariant for $\tW \tG$, to obtain
$$\mu_*^{\Phi}\tG=\frac{\mu_*^{\Psi} \tW \tG}{\mu_*^{\Psi}\tW(X)}=\frac{\mu_*^{\Psi}}{\mu_*^{\Psi}\tW(X)},$$
which, in particular, implies that
$$\mu_*^{\Phi}\tG(X)=\frac{1}{\mu_*^{\Psi}\tW(X)}.$$
Consequently, it now follows that
$$\frac{\mu_*^{\Phi}\tG}{\mu_*^{\Phi}\tG(X)}=\mu_*^{\Psi}\tW(X)\cdot \mu_*^{\Phi} \tG=\mu_*^{\Psi}.$$
This completes the proof of the theorem.
\end{proof}

\section{Proofs of the auxiliary results} \label{sec:proof_aux}

In this part of the paper we provide the proofs of Lemmas \ref{lem:Pt_prop}--\ref{lem:main1}, which have been used to prove Theorem \ref{thm:main}. 
\subsection{Proof of Lemma \ref{lem:Pt_prop}}\label{sec:proof_props}
Let us begin this section with defining
\begin{equation}
\label{def:eta}
\eta(t):=\max\{n\in\n_0:\, \tau_n\leq t\}\quad\text{for}\quad t\in\mathbb{R}_+.
\end{equation}
Obviously, $\{\eta(t)=n\}=\{\tau_n\leq t <\tau_{n+1}\}$ for any $t\in\mathbb{R}_+$ and $n\in\n_0$. 

Although \eqref{def:eta} generally does not define a Poisson process (unless $\lambda$ is constant), using assumption~\eqref{e:lambda_bounds}, one can find an upper limit of the~probability of $\{\eta(t)=n\}$ close to the corresponding probability in such a process, which is crucial in the proof of Lemma~\ref{lem:Pt_prop}. This is done in Lemma \ref{lem:tech2}, which relies on the following observation:

\begin{lemma}\label{lem:tech1}
For any $x\in X$, $t\in\mathbb{R}_+$ and $n\in\n_0$,
\begin{equation}
\label{as_ind}
\ew_{(x,0)}\left[e^{\ulambda \tau_n}\mathbbm{1}_{\{\tau_n\leq t\}} \right]\leq \frac{(\olambda t)^n}{n!}.
\end{equation}
\end{lemma}
\begin{proof}
The inequality is obvious for $n=0$. 

According to \eqref{e:dist_tau}, for every $n\in\n_0$ and each $(y,i,s)\in\bar{X}$, the conditional probability density function of $\tau_{n+1}$ given $\bar{\Phi}_n=(y,i,s)$ is of the form
$$f_{\tau_{n+1}\,|\,\bar{\Phi}_n}(t\,|\,(y,i,s))=\lambda(S_i(t-s,y)) e^{-\Lambda(y,i,t-s)}\mathbbm{1}_{[s,\infty)}(t)\quad\text{for}\quad t\in\mathbb{R}_+.$$

Let $x=(y,i)\in X$ and put $\bar{x}:=(x,0)\in\bar{X}$. We shall proceed by induction. Taking into account \eqref{e:lambda_bounds}, for $n=1$ we have

\begin{align}
\begin{split}
\label{base_step}
\ew_{\bar{x}}\left[e^{\ulambda \tau_1}\mathbbm{1}_{\{\tau_1\leq T\}} \right]&=\int_0^T e^{\ulambda t} f_{\tau_1\,|\,\bar{\Phi}_{0}}(t\,|\,\bar{x})\,dt\leq  \int_0^T e^{\ulambda t}\,\olambda e^{-\ulambda t}\,dt=\olambda T \quad\text{for}\quad T\in\mathbb{R}_+.
\end{split}
\end{align}
Now, suppose inductively that \eqref{as_ind} holds for some arbitrarily fixed $n\in\n$ and all $t\in\mathbb{R}_+$. Then, for every $T\in\mathbb{R}_+$, we can write
\begin{align*}
\ew_{\bar{x}}\left[e^{\ulambda \tau_{n+1}}\mathbbm{1}_{\{\tau_{n+1}\leq T\}}\,|\,\bar{\Phi}_n \right]&=\int_0^T e^{\ulambda t} f_{\tau_{n+1}\,|\,\bar{\Phi}_{n}}(t\,|\,\bar{\Phi}_n)\,dt\leq \int_0^T e^{\ulambda t}\, \olambda e^{-\ulambda(t-\tau_n)}\mathbbm{1}_{[\tau_n,\infty)}(t)\,dt\\
&=\olambda\int_0^T e^{\ulambda\tau_n}\mathbbm{1}_{\{\tau_n\leq t\}}\,dt.
\end{align*}
Taking the expectation of both sides of this inequality and, further, using the inductive hypothesis gives

\begin{align*}
\ew_{\bar{x}}\left[e^{\ulambda \tau_{n+1}}\mathbbm{1}_{\{\tau_{n+1}\leq T\}}\right]
&\leq \olambda\int_0^T \ew_{\bar{x}}\left[e^{\ulambda\tau_n}\mathbbm{1}_{\{\tau_n\leq t\}}\right] dt 
\leq \olambda\int_0^T \frac{(\olambda t)^n}{n!}
\,dt \\
&= \frac{\olambda^{n+1}}{n!} \int_0^T t^n\,dt=\frac{(\olambda T)^{n+1}}{(n+1)!},
\end{align*}
which ends the proof.
\end{proof}

\begin{lemma}\label{lem:tech2}
For any $x\in X$, $t\in\mathbb{R}_+$ and $n\in\n_0$,
$$\pr_{(x,0)}(\eta(t)=n)\leq e^{-\ulambda t} \frac{(\olambda t)^n}{n!}.$$
\end{lemma}
\begin{proof}
Let $x\in X$, $\bar{x}:=(x,0)$, $t\in\mathbb{R}_+$, $n\in\n_0$, and observe that
$$\pr_{\bar{x}}(\eta(t)=n)=\pr_{\bar{x}}(\tau_n\leq t <\tau_{n+1})=\ew_{\bar{x}}\left[ \mathbbm{1}_{\{\tau_n\leq t\}}\pr_{\bar{x}}(\tau_{n+1}>t\,|\, \bar{\Phi}_n)\right].$$
From \eqref{e:dist_tau} and \eqref{e:lambda_bounds} it follows that
$$\pr_{\bar{x}}(\tau_{n+1}>t\,|\,\bar{\Phi}_n)= e^{-\Lambda(Y_n,\,\xi_n,\,t-\tau_n)}\leq e^{-\ulambda(t-\tau_n)} \quad \text{on}\quad \{\tau_n\leq t\}.$$
Having this in mind and applying Lemma \ref{lem:tech1} we therefore get
$$\pr_{\bar{x}}(\eta(t)=n)\leq \ew_{\bar{x}}\left[\mathbbm{1}_{\{\tau_n\leq t\}} e^{-\ulambda(t-\tau_n)}\right]=e^{-\ulambda t}\, \ew_{\bar{x}} \left[ \mathbbm{1}_{\{\tau_n\leq t\}} e^{\ulambda\tau_n}\right]\leq e^{-\ulambda t}\frac{(\olambda t)^n}{n!},$$
which is the desired claim.
\end{proof}

Having established this, we are now in a position to conduct the announced proof.

\begin{proof}[Proof of Lemma \ref{lem:Pt_prop}]
In what follows, we will write $\bar{x}:=(x,0)$ for any given $x\in X$. 

Let $f\in B_b(X)$ and $T\in\mathbb{R}_+$. Then, appealing to \eqref{e:efp} and \eqref{def:pdmp}, for every $x\in X$, we get
\begin{align}
\begin{split}
\label{e:pt_ew}
P_T f(x)&=\ew_{\bar{x}}f(Y(T),\xi(T))=\sum_{n=0}^{\infty} \ew_{\bar{x}}\left[f(S_{\xi_n}(T-\tau_n,Y_n),\xi_n)\mathbbm{1}_{\left\{\eta(T)=n\right\}} \right]\\
&=\sum_{n=0}^{\infty} \ew_{\bar{x}}\left[\mathbbm{1}_{[0,T]}(\tau_n)\,f(S_{\xi_n}(T-\tau_n,Y_n),\xi_n)\,\mathbbm{1}_{(T,\infty)}(\tau_{n+1}) \right],
\end{split}
\end{align}
with $\eta(\cdot)$ defined by \eqref{def:eta}. On the other hand, in view of \eqref{e:canonical}, it is clear that, for any~$x\in X$, $g,h\in B_b(\bar{X})$ and $n\in\n_0$,
\begin{align}
\begin{split}
\label{e:dod}
\ew_{\bar{x}}[g(\bar{\Phi}_n)h(\bar{\Phi}_{n+1})]&=\int_{\bar{X}}\int_{\bar{X}} g(w)h(z)\,\bar{P}(w,dz)\bar{P}^n(\bar{x},dw)\\
&=\int_{\bar{X}}g(w)\bar{P}h(w)\bar{P}^n(\bar{x},dw)=\bar{P}^n(g\bar{P}h)(\bar{x}).
\end{split}
\end{align}
Hence, we can write
\begin{align}
\begin{split}
\label{e:decomp}
P_T f(x)&=\sum_{n=0}^{\infty}\ew_{\bar{x}}[g_T(Y_n,\xi_n,\tau_n)h_T(Y_{n+1},\xi_{n+1},\tau_{n+1})]\\&=\sum_{n=0}^{\infty}\bar{P}^n(g_T\bar{P}h_T)(\bar{x})\quad\text{for every}\quad x\in X,
\end{split}
\end{align}
where $g_T,h_T\in B_b(\bar{X})$ are given by
\begin{equation}
\label{def:gh}
g_T(u,j,s):=\mathbbm{1}_{[0,T]}(s)f(S_j(T-s,u),j)\quad\text{and}\quad
h_T(u,j,s):=\mathbbm{1}_{(T,\infty)}(s)
\end{equation}
for any $u\in Y$, $j\in I$ and $s\in\mathbb{R}_+$.

\textbf{\ref{lem:i}}: Suppose that $J$ is Feller, $f\in C_b(X)$ and let $T\in\mathbb{R}_+$. The aim is to prove that $P_T f$ is continuous.

Let us first observe that, for any function $\varphi\in B_b(\bar{X})$ such that $X\ni x\mapsto\varphi(x,s)$ is continuous for every $s\geq 0$ (which is obviously the case for $g_T$ and $h_T$), the map \hbox{$X\ni x\mapsto \bar{P}\varphi(x,s)$} is continuous for each $s\geq 0$ as well. Indeed, fix $s\in\mathbb{R}_+$, and note that, for any \hbox{$i,j\in I$} and \hbox{$t\in\mathbb{R}_+$}, the map $Y \ni y \mapsto J(\pi_{ij} \varphi(\cdot,j,t))(y)$ is continuous since the kernel $J$ is Feller, and~\hbox{$\pi_{ij},\, \varphi(\cdot,j,t)\in C_b(Y)$}. Consequently, taking into account the continuity of both $\lambda$  and the semiflows, it follows that, for all $j\in I$ and $t\geq 0$, the maps 
$$X\ni (y,i) \mapsto \lambda(S_i(t,y)) e^{-\Lambda(y,i,t)}J(\pi_{ij}\varphi(\cdot,j,s+t))(S_i(t,y))$$
are continuous. Moreover, all these maps are bounded by $\olambda\, e^{-\ulambda t}\norma{\varphi}_{\infty}$. Hence, using the  Lebesgue dominated convergence theorem, we can conclude that the function
$$X\ni (y,i)=x \mapsto \bar{P}\varphi(x,s)=\sum_{j\in I}\int_0^{\infty} \lambda(S_i(t,y)) e^{-\Lambda(y,i,t)}\, J(\pi_{ij} \varphi(\cdot,j,s+t))(S_i(t,y))\,dt$$
is continuous, as claimed.

In light of the observation above, all the maps $X\ni x\mapsto \bar{P}^n(g_T\bar{P}h_T)(\bar{x})$, $n\in \n_0$, are continuous. Furthermore, from \eqref{e:dod} and Lemma \ref{lem:tech2} it follows that
\begin{align}
\begin{split}
\label{e:com_bound}
|\bar{P}^n(g_T\bar{P}h_T)(\bar{x})|&=\left|\ew_{\bar{x}}[f(S_{\xi_n}(T-\tau_n,Y_n),\xi_n))\mathbbm{1}_{\{\eta(T)=n\}} ]\right|\\
&\leq \norma{f}_{\infty}\pr_{\bar{x}}(\eta(t)=n)\leq \norma{f}_{\infty}e^{-\ulambda T}\frac{(\olambda T)^n}{n!}
\end{split}
\end{align}
for all $x\in X$ and $n\in\n_0$. Hence, the continuity of $P_Tf$ can be now deduced by applying the discrete version of the Lebesgue dominated convergence theorem to \eqref{e:decomp}.
\vspace{0.2cm}

\textbf{\ref{lem:ii}}: Let us define
$$u_f(x,T):=\sum_{n=2}^{\infty}\bar{P}^n(g_T\bar{P}h_T)(\bar{x})\quad\text{for}\quad x\in X,\; T\in\mathbb{R}_+,\vspace{-0.5cm}$$
and
\begin{align*}
\psi_f(x,t,T)&:=\lambda(S_i(t,y))e^{-\Lambda(y,i,t)}\int_Y\sum_{j\in I}  \pi_{ij}(u) e^{-\Lambda(u,j,T-t)}\\
&\quad\times f(S_j(T-t,u),j) J(S_i(t,y),du)\quad\text{for}\quad x=(y,i)\in X,\; (t,T)\in \dpsi.
\end{align*}
Obviously, $u_f(\cdot, T)$ and $\psi_f(\cdot, t, T)$ are Borel measurable for any $T\in\mathbb{R}_+$ and $0\leq t\leq T$. 

Now, fix $T\in\mathbb{R}_+$. Then, according to \eqref{e:decomp}, we have
\begin{equation}
\label{e:P_T_app1} 
P_T f(x)=g_T\bar{P}h_T(\bar{x})+\bar{P}(g_T\bar{P}h_T)(\bar{x})+u_f(x,T)\quad\text{for}\quad x\in X.
\end{equation}
Keeping in mind \eqref{def:P_bar} and \eqref{def:gh}, we see that for any $x=(y,i)\in X$ and $s\in\mathbb{R}_+$,
\begin{align}
\begin{split}
\label{e:gph}
g_T\bar{P}h_T(x,s)&=\mathbbm{1}_{[0,T]}(s)f(S_i(T-s,y),i)\int_0^{\infty} \lambda(S_i(t,y)) e^{-\Lambda(y,i,t)} \mathbbm{1}_{(T,\infty)}(s+t)\,dt
\\
&=\mathbbm{1}_{[0,T]}(s) e^{-\Lambda(y,\,i,\,T-s)} f(S_i(T-s,y),i),
\end{split}
\end{align}
which, in particular, gives
\begin{equation}
\label{e:P_T_app2}
g_T\bar{P}h_T(\bar{x})=g_T\bar{P}h_T(x,0)= e^{-\Lambda(y,i,T)} f(S_i(T,y),i) .
\end{equation}
Appealing to \eqref{e:gph}, we can also conclude that
\begin{align*}
\bar{P}&(g_T\bar{P}h_T)(\bar{x})=\int_0^{\infty} \lambda(S_i(t,y)) e^{-\Lambda(y,i,t)} \int_Y \sum_{j\in I} \pi_{ij}(u)(g_T\bar{P}h_T)(u,j,t)\,J(S_i(t,y),du)\,dt\\
&=\int_0^T\lambda(S_i(t,y)) e^{-\Lambda(y,i,t)} \int_Y \sum_{j\in I} \pi_{ij}(u) e^{-\Lambda(u,\,j,\,T-t)} f(S_j(T-t,u),j) \,J(S_i(t,y),du)\,dt\\
&=\int_0^T \psi_f((y,i),t,T)\,dt \quad\text{for all}\quad x=(y,i)\in X,
\end{align*}
which, together with \eqref{e:P_T_app1} and \eqref{e:P_T_app2} implies \eqref{e:P_T_app}.

Further, referring to \eqref{e:com_bound}, we obtain
\begin{align}
\begin{split}
\label{e:est_rest}
\frac{|u_f(x,T)|}{T}&\leq \norma{f}_{\infty}\frac{1}{T}e^{-\ulambda T}\sum_{n=2}^{\infty} \frac{(\olambda T)^n}{n!}=\norma{f}_{\infty}\frac{1}{T}e^{-\ulambda T}(e^{\olambda T}-1-\olambda T)\\
&=\norma{f}_{\infty} e^{-\ulambda T}\left(\frac{e^{\olambda T}-1}{T}-\olambda  \right) \quad \text{for all}\quad x\in X,\; T\in\mathbb{R}_+,
\end{split}
\end{align}
which shows that $u_f$ fulfills the conditions specified in~\eqref{e:prop_uf}. In turn, the properties of the function~$\psi_f$ stated in \eqref{e:psi_prop1} and \eqref{e:psi_prop2} follow directly from its definition and~\eqref{e:lambda_bounds}.

Now, suppose that $J$ satisfies Assumption \ref{cnd:f}, $f\in C_b(X)$, and let  $x=(y,i)\in X$. To prove that the map $\dpsi \ni (t,T)\mapsto \psi(x,t,T)$ is jointly continuous, define
$$g(u,t):=\sum_{j\in I} \pi_{ij}(u) e^{-\Lambda(u,j,t)} f(S_j(t,u),j) \quad\text{for}\quad u\in Y,\;t\in\mathbb{R}_+.$$
Then
$$\psi_f(x,t,T)=\lambda(S_i(t,y))e^{-\Lambda(y,i,t)} Jg(\cdot,T-t)(S_i(t,y))\quad\text{for}\quad (t,T)\in \dpsi.$$
Taking into account \eqref{e:lambda_bounds}, for every $j\in I$ and any $(u,t), (u_0,t_0)\in Y\times\mathbb{R}_+$, we have
$$|\Lambda(u,j,t)-\Lambda(u_0,j,t_0)|\leq \int_0^{t_0} |\lambda(S_j(h,u))-\lambda(S_j(h,u_0)) |\,dh +\olambda |t-t_0|.$$
Hence, in view of the continuity of $\lambda$ and $S_j(h,\cdot)$ for any $h\geq 0$, we can conclude (by applying the Lebesgue dominated convergence theorem) that $(u,t)\mapsto \Lambda(u,j,t)$ is jointly continuous  for each $j\in I$. This, together with the continuity of $f$, $S_j$ and $\pi_{ij}$, \hbox{$j\in J$}, shows that $g$ is jointly continuous as well, which, in turn, guarantees that $g\in C_b(Y\times\mathbb{R}_+)$, since $|g(u,t)|\leq \norma{f}_{\infty}$ for all $(u,t)\in Y\times\mathbb{R}_+$. Eventually, it now follows from Assumption \ref{cnd:f} and the continuity of $S_i(\cdot,y)$ that the map 
$\dpsi\ni(t,T)\mapsto Jg(\cdot, T-t)(S_i(t,y))$ is jointly continuous, and thus so is $\dpsi\ni (t,T)\mapsto \psi_f(x,t,T)$.

\vspace{0.2cm}
\textbf{\ref{lem:iii}}: Statement \ref{lem:iii} follows immediately from \ref{lem:ii}.
\end{proof}

\subsection{Proofs of Lemmas \ref{lem:main2} and \ref{lem:main1}} \label{sec:proof_lems}
Before we proceed, let us recall that $A_P$ and $A_Q$ stand for the weak generators of $\{P_t\}_{t\in\mathbb{R}_+}$ and $\{Q_t\}_{t\in\mathbb{R}_+}$, respectively, considered as contraction semigroups on $C_b(X)$, where $\{Q_t\}_{t\in\mathbb{R}_+}$ is defined by \eqref{def:Qt}. Also, keep in mind that $G$ and $W$ are the kernels specified in \eqref{def:G} and~\eqref{def:W}, respectively, while $\hat{\lambda}$ is given by \eqref{def:hlambda}.
\begin{proof}[Proof of Lemma \ref{lem:main2}]
Let $f\in C_b(X)$ and define
\begin{equation}
\label{def:delta}
\Delta_f(x,T):=\frac{1}{T}\int_0^T \psi_f(x,t,T)\,dt-\frac{1-e^{-\Lambda(x,T)}}{T}f(x)+\frac{u_f(x,T)}{T}
\end{equation}
for any $x\in X$ and $T>0$, where $\psi_f$ is the function specified in statement \ref{lem:ii} of Lemma~\ref{lem:Pt_prop}. Recall that, according to this lemma, the map $\dpsi\ni (t,T) \mapsto \psi_f(x,t,T)$ is jointly continuous for every $x\in X$ (due to Assumption \ref{cnd:f}).

We will first show that
\begin{equation}
\label{e:wlim1}
\wlim_{T\to 0} \Delta_f(\cdot,T)=\hlambda Wf-\hlambda f\in C_b(X).
\end{equation}
To do this, let $x\in X$. Since the maps $[0,T]\ni t\mapsto \psi_f(x,t,T)$, $T>0$, are continuous, the mean value theorem for integrals guarantees that, for every $T>0$, the first component on the right-hand side of \eqref{def:delta} is equal to $\psi_f(x,t_x(T),T)$ for some $t_x(T)\in [0,T]$. Consequently, taking into account the continuity of $T\mapsto \psi_f(x,t_x(T),T)$ and \eqref{e:psi_prop2}, we see that
$$\lim_{T\to 0}\frac{1}{T}\int_0^T \psi_f(x,t,T)\,dt=\lim_{T\to 0}\psi_f(x,t_x(T),T) =\psi_f(x,0,0)=\hlambda(x)Wf(x).$$
This, together with the fact that
$$\norma{\frac{1}{T}\int_0^T \psi_f(\cdot,t,T)\,dt }_{\infty}\leq \olambda\norma{f}_{\infty}\quad\text{for all}\quad T>0,$$
resulting from \eqref{e:psi_prop1}, implies that $\wlim_{T\to 0}T^{-1}\int_0^T \psi_f(\cdot,t,T)\,dt=\hlambda Wf$. Further, from l'Hospital's rule it follows that 
$$\lim_{T\to 0} \frac{1-e^{-\Lambda(x,T)}}{T}=\hlambda(x),$$
and \eqref{e:lambda_bounds} ensures that
$$\norma{\frac{1-e^{-\Lambda(\cdot,T)}}{T}}_{\infty}\leq \frac{1-e^{-\olambda T}}{T}\leq \olambda+1\quad\text{for sufficiently small}\quad T>0,$$
whence $\wlim_{T\to 0} (1-e^{-\Lambda(\cdot,T)})/T=\hlambda$. Moreover, \eqref{e:prop_uf} gives $\wlim_{T\to 0} u_f(\cdot,T)/T=0$. Finally, we see $\hlambda Wf-\hlambda f\in C_b(X)$ since $W$ is Feller (by Assumption \ref{cnd:f}) and \hbox{$\hlambda\in C_b(X)$}. This all proves that \eqref{e:wlim1} holds. 

Now, referring to \eqref{e:P_T_app}, we can write
\begin{align*}
\frac{P_T f(x)-f(x)}{T}&=e^{-\Lambda(y,i,T)}\cdot \frac{f(S_i(t,y),i)-f(y,i)}{T}+\frac{1}{T}\int_0^T \psi_f((y,i),t,T)\,dt\\
&\quad-\frac{1-e^{-\Lambda(y,i,T)}}{T} f(y,i)+\frac{u_f((y,i),T)}{T}\quad \text{for}\quad x=(y,i)\in X,\;T>0,
\end{align*}
which, due to \eqref{def:Qt} and \eqref{def:delta}, gives
$$
\frac{P_T f- f}{T}=e^{-\Lambda(\cdot,T)}\cdot\frac{Q_T f -f}{T}+\Delta_f(\cdot,T)\quad\text{for}\quad T>0.
$$
In view of \eqref{e:wlim1} and the fact that $\wlim_{T\to 0} e^{-\Lambda(\cdot,T)}=1$, this observation shows that \hbox{$f\in D(A_P)$} iff $f\in D(A_Q)$, and that \eqref{e:APQ} is satisfied for all $f\in D(A_P)=D(A_Q)$. The proof of the lemma is therefore complete.
\end{proof}

\begin{proof}[Proof of Lemma \ref{lem:main1}]
Obviously, it suffices to show that \eqref{e:inv1} and \eqref{e:inv2} hold.

To this end, let $f\in C_b(X)$. Then, using the flow property, for any $x=(y,i)\in X$ and any $T>0$, we obtain
\begin{align*}
Q_T G f(x)&= Gf(S_i(T,y),i)\\&=\int_0^{\infty} \lambda(S_i(\bar{t}+T,y))\exp\left(-\int_0^{\bar{t}} \lambda(S_i(\bar{h}+T,y))\,d\bar{h}\right)f(S_i(\bar{t}+T,y),i)\,d\bar{t}.
\end{align*}
Making the substitutions $t=\bar{t}+T$ and $h=\bar{h}+T$ therefore gives
\begin{align*}
Q_T G f(x)&=\int_T^{\infty} \lambda(S_i(t,y))\exp\left(-\int_T^{t} \lambda(S_i(h,y))\,dh\right)f(S_i(t,y),i)\,dt.
\end{align*}
Further, since
$$\int_T^{t} \lambda(S_i(h,y))\,dh=L(y,i,t)-L(y,i,T)\quad\text{for}\quad t\geq T,$$
it follows that
\begin{align*}
Q_T G f(x)&=e^{L(y,i,T)}\int_T^{\infty} \lambda(S_i(t,y))e^{-L(y,i,t)}f(S_i(t,y),i)\,dt\\
&=e^{L(y,i,T)}\left(Gf(y,i)- \int_0^T \lambda(S_i(t,y))e^{-L(y,i,t)}f(S_i(t,y),i)\,dt \right).
\end{align*}
Hence, for all $x=(y,i)\in X$ and $T>0$, we have
\begin{align*}
\frac{Q_T(Gf)(x)-Gf(x)}{T}&=\frac{e^{L(y,i,T)}-1}{T}Gf(y,i)\\ 
&\quad-e^{L(y,i,T)}\frac{1}{T}\int_0^T \lambda(S_i(t,y))e^{-L(y,i,t)}f(S_i(t,y),i)\,dt.
\end{align*}
Now, using l'Hospital's rule and taking into account the continuity of the integrand on the right-hand side of this equality, we can conclude that
\begin{align*}
\lim_{T\to 0}\frac{Q_T(Gf)(x)-Gf(x)}{T}&=\lambda(y)Gf(y,i)-\lambda(y)f(y,i)\\\
&=\hlambda(x)(Gf(x)-f(x)) \quad\text{for every}\quad x=(y,i)\in X,
\end{align*}
and $\hlambda(Gf-f)\in C_b(X)$, since $G$ is Feller. Moreover, bearing in mind \eqref{e:lambda_bounds}, we see that
\begin{align*}
\norma{\frac{Q_T(Gf)-Gf}{T}}_{\infty}\leq \left(\frac{e^{\olambda T}-1}{T} + \olambda e^{\olambda T} \right)\norma{f}_{\infty}\leq \left(\olambda+1+\olambda e^{\olambda \delta}\right)\norma{f}
\end{align*}
for every $T\in(0,\delta)$ with sufficiently small $\delta>0$. We have therefore shown that
$$\wlim_{T\to 0} \frac{Q_T(Gf)-Gf}{T}=\hlambda(Gf-f)\in C_b(X)=L_{0,Q}.$$
This obviously means that $Gf\in D(A_Q)$ and $A_Q(Gf)=\hlambda(Gf-f)$,
which immediately implies \eqref{e:inv1}.

What is left is to prove \eqref{e:inv2}. To this end, let $f\in D(A_Q)$ and $x=(y,i)\in X$. Then
\begin{align*}
Gf(x)&=\int_0^{\infty} \lambda(S_i(t,y))e^{-\Lambda(y,i,t)}f(S_i(t,y),i)\,dt=-
\int_0^{\infty} \left(\frac{d^+}{dt}e^{-\Lambda(y,i,t)}\right)Q_tf(y,i)\,dt\\
&=-\left[e^{-\Lambda(y,i,t)} f(S_i(t,y),i)\right]_0^{\infty}
+\int_0^{\infty} e^{-\Lambda(y,i,t)} \left(\frac{d^+}{dt} Q_t f(y,i) \right)\,dt\\
&=f(y,i)+\int_0^{\infty} e^{-\Lambda(y,i,t)} \left(\frac{d^+}{dt} Q_t f(y,i) \right)\,dt.
\end{align*}
According to statement \ref{rem:ii} of Remark \ref{rem:inf}, we have
$$\frac{d^+}{dt} Q_t f(y,i)=Q_t A_Q f(y,i)= A_Qf(S_i(t,y),i)\quad\text{for all}\quad t\geq 0,$$
whence it finally follows that
\begin{align*}
Gf(x)&=f(y,i)+\int_0^{\infty} \lambda(S_i(t,y)) e^{-\Lambda(y,i,t)}\frac{A_Q f(S_i(t,y),i)}{\lambda(S_i(t,y))}\,dt=f(x)+G(A_Q/\hlambda)(x),
\end{align*}
which clearly yields \eqref{e:inv2}. The proof is now complete.
\end{proof}

\section{Application to a particular subclass of PDMPs}\label{sec:application}
In this section we shall use Theorem \ref{thm:main} and \hbox{\cite[Theorem 4.1]{b:czapla_kubieniec}} (cf.~also \hbox{\cite[Theorem 4.1]{b:czapla_erg}} for the case of constant $\lambda$) to provide a set of tractable conditions guaranteeing the existence and uniqueness of a stationary distribution for some particular PDMP, where~$J$ is the transition law of a random iterated function system with an arbitrary family of transformations and state-dependent probabilities of selecting them.

Let $\Theta$ be a topological space, and suppose that $\vartheta$ is a non-negative Borel measure on~$\Theta$. Further, consider an arbitrary set \hbox{$\{w_{\theta}:\,\theta\in \Theta\}$} of continuous transformations from $Y$ to itself and an associated collection $\{p_{\theta}:\, \theta\in\Theta\}$ of continuous maps from $Y$ to $\mathbb{R}_+$ such that, for every $y\in Y$, $\theta\mapsto p_{\theta}(y)$ is a probability density function with respect to $\vartheta$. Moreover, assume that $(y,\theta)\mapsto w_{\theta}(y)$ and $(y,\theta)\mapsto p_{\theta}(y)$ are product measurable.  Given this framework, we are concerned with the kernel $J$ defined by 
\begin{equation}
\label{e:ifs}
J(y,B)=\int_{\Theta}\mathbbm{1}_B(w_{\theta}(y))\,p_{\theta}(y)\,\vartheta(d\theta)\quad\text{for all}\;\; y\in Y\;\text{and}\;\; B\in \mathcal{B}(Y).
\end{equation}
The transition law $P$, specified by \eqref{def:P}, can be then expressed exactly as in \cite{b:czapla_kubieniec}, i.e.,
\begin{align}
\begin{split}
\label{def:P_ifs}
P((y,i), A)&= \int_0^{\infty} \lambda(S_i(t,y)) e^{-\Lambda(y,i,t)} \int_{\Theta} p_{\theta}(S_i(t,y))\\
&\quad \times \sum_{j\in I}\ \pi_{ij}(w_{\theta}(S_i(t,y)))\mathbbm{1}_A(w_{\theta}(S_i(t,y)),j)\,\vartheta(d\theta)\,dt
\end{split}
\end{align}
for any $y\in Y$, $i\in I$ and $A\in\mathcal{B}(X)$. Moreover, note that, in this case, the first coordinate of $\Phi$ satisfies the recursive formula:
$$Y_{n+1}=w_{\eta_{n+1}}(Y(\tau_{n+1}-))=w_{\eta_{n+1}}(S_{\xi_n}(\Delta\tau_{n+1},Y_n))\;\;\text{a.s.}\quad \text{for}\quad n\in\n_0,$$
where $\{\eta_n\}_{n\in\n}$ is a sequence of random variables with values in $\Theta$, such that
$$
\pr(\eta_{n+1}\in D\,|\,S_{\xi_n}(\Delta\tau_{n+1},Y_n)=y)=\int_D p_{\theta}(y)\vartheta(d\theta)\;\;\text{for}\;\; D\in\mathcal{B}(\Theta),\;y\in Y,\;n\in\n_0.
$$

The aforementioned \cite[Theorem 4.1]{b:czapla_kubieniec} (whose proof relies on \cite[Theorem 2.1]{b:kapica}) provides certain conditions under which the transition operator $P$, induced by \eqref{def:P_ifs}, is exponentially ergodic (in the so-called bounded Lipschitz distance) and notably possesses a unique invariant distribution, which belongs to $\mathcal{M}_1^V(X)$. To establish the main result of this section, we shall therefore incorporate several additional requirements, which, along with Assumptions \ref{cnd:a1} and \ref{cnd:a2}, will enable us to apply this theorem and, simultaneously, ensure the fulfillment of Assumption \ref{cnd:f}, involved in Theorem~\ref{thm:main}.

In connection with the above, we first make one more assumption on the semiflows $S_i$:
\begin{assumption}\label{cnd:a3}
There exist a Lebesgue measurable function $\varphi:\mathbb{R}_+\to\mathbb{R}_+$ satisfying 
\begin{equation}
\label{e:phi}
\int_0^{\infty}e^{-\underline{\lambda} t} \varphi(t)\,dt<\infty,
\end{equation} 
and a function $\mathcal{L}:Y\to\mathbb{R}_+$, bounded on bounded sets, such that
$$\rho_Y(S_i(t,y),S_j(t,y))\leq \varphi(t)\mathcal{L}(y)\quad\text{for any}\quad t\geq 0,\;y\in Y,\;i,j\in I.$$
\end{assumption}
Further, we employ the following assumption on the probabilities $\pi_{ij}$, associated with the semiflows switching:
\begin{assumption}\label{cnd:pi}
There exist positive constants $L_{\pi}$ and  $\delta_{\pi}$ such that 
\begin{gather*}
\sum_{j\in I}|\pi_{ij}(u)-\pi_{ij}(v)|\leq L_{\pi}\rho_Y(u,v) \quad \text{for any} \quad u,v \in Y,\; i\in I;\\
\sum_{j\in I}\min\{\pi_{ij}(u),  \pi_{kj}(v)  \}\geq \delta_{\pi} \quad \text{for any} \quad u,v \in Y,\; i,k\in I.
\end{gather*}
\end{assumption}
Finally, let us impose certain conditions on the components of the kernel~$J$, defined by \eqref{e:ifs}.
\begin{assumption} \label{cnd:w1}
There exist $y_*\in Y$ such that
$$
\gamma(y_*):=\sup_{y\in Y}\int_{\Theta} \rho_Y(w_{\theta}(y_*),y_*)p_{\theta}(y)\,\vartheta(d\theta)<\infty.
$$
\end{assumption}
\begin{assumption} \label{cnd:w}
There exist positive constants $L_w$, $L_p>0$ and $\delta_p>0$ such that
\begin{gather}
\label{cnd:w2}\int_{\Theta} \rho_Y(w_{\theta}(u),w_{\theta}(v))\,p_{\theta}(u)\vartheta(d\theta)\leq L_w\rho_Y(u,v)\quad \text{for any}\quad u,v\in Y;\\
\label{cnd:w3}\int_{\Theta} |p_{\theta}(u)-p_{\theta}(v)|\,\vartheta(d\theta)\leq L_p\rho_Y(u,v)\quad\text{for any}\quad u,v\in Y;\\
\nonumber\int_{\Theta(u,v)} \min\{p_{\theta}(u)\wedge p_{\theta}(v)\}\,\vartheta(d\theta)\geq \delta_p\quad\text{for any}\quad u,v\in Y,
\end{gather}
where 
$$
\Theta(u,v):=\{\theta\in \Theta:\,\rho_Y(w_{\theta}(u),w_{\theta}(v))\leq L_w\rho_Y(u,v)\}.$$
\end{assumption}
In addition to that, we will require that the constants $L,\alpha$ in Assumption \ref{cnd:a2} and $L_w$ in Assumption \ref{cnd:w} are interrelated by the inequality
\begin{equation}
\label{e:balance}
LL_w\olambda+\alpha<\ulambda.
\end{equation}

It is worth stressing here that conditions similar in form to those gathered in Assumptions~\ref{cnd:w1} and \ref{cnd:w} (with $L_w<1$) are commonly \hbox{utilized} when examining the existence of invariant measures and stability properties of random iterated function systems; See, e.g., \cite[Proposition 3.1]{b:kapica} or \hbox{\cite[Theorem 3.1]{b:szarek}}.

\begin{remark}\label{rem:usage} In what follows, we shall demonstrate that \cite[Theorem 4.1]{b:czapla_kubieniec}, involving conditions \cite[(A1)-(A6) and (20)]{b:czapla_kubieniec}, remains valid (with almost the same proof) under \hbox{Assumptions} \ref{cnd:a1}, \ref{cnd:a2} and \hbox{\ref{cnd:a3}--\ref{cnd:w}}, provided that the constants $L,\alpha$ and $L_w$ satisfy \eqref{e:balance} (which corresponds to (20) in \cite{b:czapla_kubieniec}), and that  $\lambda$ is Lipschitz continuous. 
\begin{itemize}

\item[(i)]  Hypothesis (A1), which states that, for some $y_*\in Y$,
$$\max_{i\in I}\sup_{y\in Y}\int_0^{\infty} e^{-\ulambda t}\int_{\Theta}\rho_Y(w_{\theta}(S_i(t,y_*)),y_*)p_{\theta}(S_i(t,y))\,\vartheta(d\theta)\,dt<\infty,$$
can be replaced with the conjunction of Assumptions \ref{cnd:a1} and \ref{cnd:w1}. More precisely, the only place where (A1) is used in the proof of \cite[Theorem 4.1]{b:czapla_kubieniec} occurs in Step 1, which aims to show that $P$ enjoys the Lyapunov condition with $V$ given by \eqref{def:V}, i.e., there exist $a\in [0,1)$ and $b\geq 0$ such that
\begin{equation}
\label{e:lapunov_V}
PV(x)\leq a V(x)+b\quad\text{for all}\quad x\in X.
\end{equation}
This condition is achieved there with\vspace{-0.08cm}
\begin{equation}\label{def:a}
a:=\frac{\olambda L_w L}{\ulambda-\alpha},
\end{equation}
which is obviously $<1$ due to inequality \eqref{e:balance}.
However, as will be seen in Lemma~\ref{lem:lapunov_P} (established  below), condition \eqref{e:lapunov_V} can be derived (with the same $a$) using  Assumptions  \ref{cnd:a1}, \ref{cnd:a2}, \ref{cnd:w1}, and property \eqref{cnd:w2} (included in Assumption \ref{cnd:w}).

\item[(ii)] Condition (A2) is equivalent to the conjunction of Assumptions \ref{cnd:a2} and \ref{cnd:a3} with \hbox{$\varphi(t)=t$}. Nevertheless, a simply analysis of the proof of \cite[Theorem 4.1]{b:czapla_kubieniec} shows that it remains valid if $\varphi:\mathbb{R}_+\to\mathbb{R}_+$ is an arbitrary Lebesgue measurable function satisfying~\eqref{e:phi}.

\item[(iii)] Hypotheses (A3), (A5) and (A6) simply coincide with Assumptions \ref{cnd:pi} and \ref{cnd:w}.

\item[(iv)]  Hypothesis (A4) just corresponds to the assumed Lipschitz continuity of $\lambda$.
\end{itemize}
\end{remark}

\begin{remark}
Regarding Remark \ref{rem:usage}(i), the sole reason we employ Assumptions \ref{cnd:a1} and~\ref{cnd:w1} in this paper instead of hypothesis \cite[(A1)]{b:czapla_kubieniec} is to get the conclusion of Proposition~\ref{prop:m1v}(i), which ensures that the invariant measures of $\{P_t\}_{t\in\mathbb{R_+}}$ inherit the property of having finite first moments w.r.t. $V$ from the invariant measures of $P$.

It is worth emphasizing that Assumptions \ref{cnd:a1} and \ref{cnd:w1} do not generally imply \cite[(A1)]{b:czapla_kubieniec}. However, it is easy to check (using the triangle inequality) that such an implication does hold in at least three special cases: first, when all $w_{\theta}$ are Lipschitz continuous (thus strengthening \eqref{cnd:w2}); second, when all the maps $p_{\theta}$ are constant, and third, if there exists a point $y_*\in Y$ such that $S_i(t,y_*)=y_*$ for all $t\geq 0$ and $i\in I$ (in this case Assumptions \ref{cnd:a1} holds trivially).
\end{remark}

The following two lemmas justifies statement (i) of Remark \ref{rem:usage}.

\begin{lemma}\label{lem:lapunov_J}
Suppose that Assumption \ref{cnd:w1} is fulfilled, and that condition \eqref{cnd:w2} holds for some $L_w>0$. Then the operator $J(\cdot)$~on~$B_b(X)$ induced by the kernel specified in \eqref{e:ifs} satisfies \eqref{e:lapunov_J} with $\widetilde{V}:=\rho_Y(\cdot,y_*)$, $\tilde{a}:=L_w$ and $\tilde{b}:=\gamma(y_*)$.
\end{lemma}
\begin{proof}
For every $y\in Y$ we have
\begin{align*}
J\widetilde{V}(y)&=\int_{\Theta} \rho_Y(w_{\theta}(y),y_*)p_{\theta}(y)\vartheta(d\theta)\\
&\leq \int_{\Theta} \rho_Y(w_{\theta}(y),w_{\theta}(y_*))p_{\theta}(y)\vartheta(d\theta)+ \int_{\Theta} \rho_Y(w_{\theta}(y_*),y_*)p_{\theta}(y)\vartheta(d\theta)\\
&\leq L_w\rho_Y(y,y_*)+\gamma(y_*)=L_w\widetilde{V}(y)+\gamma(y_*).
\end{align*}
\end{proof}

\begin{lemma}\label{lem:lapunov_P}
Under Assumptions \ref{cnd:a1} and \ref{cnd:a2}, if  \eqref{e:lapunov_J} holds with $\widetilde{V}:=\rho_Y(\cdot,y_*)$ and certain constants $\tilde{a},\tilde{b}\geq 0$, then \eqref{e:lapunov_V} is satisfied with $V$ given by \eqref{def:V} and the constants 
$$a:=\frac{\olambda \tilde{a} L}{\ulambda-\alpha} \quad \text{and} \quad b:=\olambda\left(\tilde{a} \beta(y_*)+ \tilde{b}\ulambda^{-1}\right).$$
In particular, if $J$ is of the form \eqref{e:ifs}, then Assumptions \ref{cnd:a1}, \ref{cnd:a2}, \ref{cnd:w1}  and condition  \eqref{cnd:w2} yield that \eqref{e:lapunov_V} holds with the constant $a$ given by \eqref{def:a}.

\end{lemma}

\begin{proof}
Clearly, $b<\infty$ by Assumption \ref{cnd:a1}. Let $x=(y,i)\in X$. Using condition \eqref{e:lapunov_J} and Assumption~\ref{cnd:a2}, we obtain
\begin{align*}
J\widetilde{V}(S_i(t,y))&\leq \tilde{a} \rho_Y(S_i(t,y),y_*)+\tilde{b}\\
& \leq \tilde{a} \left(\rho_Y(S_i(t,y),S_i(t,y_*))+\rho_Y(S_i(t,y_*),y_*) \right)+\tilde{b}\\
&\leq \tilde{a} Le^{\alpha t} \rho_Y(y,y_*)+\tilde{a}\rho_Y(S_i(t,y_*),y_*)+\tilde{b}\quad\text{for any}\quad t\in\mathbb{R}_+.
\end{align*}
This, in turn, gives
\begin{align*}
PV(x)&=\int_0^{\infty} \lambda(S_i(t,y))e^{-\Lambda(y,i,t)} J\widetilde{V}(S_i(t,y))\,dt\\
&\leq \olambda \int_0^{\infty} e^{-\ulambda t}\left(\tilde{a} Le^{\alpha t} \rho_Y(y,y_*)+\tilde{a}\rho_Y(S_i(t,y_*),y_*)+\tilde{b} \right)\,dt\\
&=\olambda\tilde{a}L\left(\int_0^{\infty} e^{-(\ulambda-\alpha)t}\,dt\right)\widetilde{V}(y)+\olambda\left(\tilde{a}\int_0^{\infty} e^{-\ulambda t}\rho_Y(S_i(t,y_*),y_*)\,dt+ \tilde{b}\ulambda^{-1} \right)\\
&\leq \frac{\olambda \tilde{a} L}{\ulambda-\alpha} V(x)+\olambda\left(\tilde{a}\beta(y_*)+\tilde{b}\ulambda^{-1} \right).
\end{align*}
The second part of the assertion follows directly from Lemma \ref{lem:lapunov_J}.
\end{proof}

What is now left to establish the main theorem of this section is to show that the model under consideration fulfills Assumption \ref{cnd:f}.

\begin{lemma}\label{lem:f}
If condition \eqref{cnd:w3} holds for some $L_p>0$, then the kernel $J$ defined by \eqref{e:ifs} fulfills Assumption \ref{cnd:f}. 
\end{lemma}
\begin{proof}
Let $g\in C_b(Y\times\mathbb{R}_+)$ and fix \hbox{$(y_0,t_0)\in Y\times\mathbb{R}_+$}. Then, for any $y\in Y$ and $t\geq 0$, we can write
\begin{align*}
|Jg(\cdot,t_0)(y_0) - Jg(\cdot,t)(y)|&\leq\int_{\Theta}|g(w_{\theta}(y_0),t_0)p_{\theta}(y_0)-g(w_{\theta}(y),t)p_{\theta}(y)|\,\vartheta(d\theta)\\
&\leq\int_{\Theta}|g(w_{\theta}(y_0),t_0)-g(w_{\theta}(y),t)|p_{\theta}(y_0)\,\vartheta(d\theta)\\
&\quad+ \int_{\Theta} |g(w_{\theta}(y),t)|\cdot|p_{\theta}(y_0)-p_{\theta}(y)|\,\vartheta(d\theta).
\end{align*}
It now suffices to observe that both terms on the right-hand side of this estimation tend to~$0$ as $(t,y)\to (t_0,y_0)$. The convergence of the first term follows from the Lebesgue dominated convergence theorem, since $w_{\theta}$ and $g$ are continuous (and the latter is also bounded). The second one converges by condition \eqref{cnd:w3} and the boundedness of $g$, which enables estimating it from above by $\norma{g}_{\infty}L_p\rho_Y(y_0,y)$.
\end{proof}

\begin{theorem}\label{thm:app}
Let $J$ be of the form \eqref{e:ifs}. Further, suppose that Assumptions \ref{cnd:a1},  \ref{cnd:a2} and \ref{cnd:a3}--\ref{cnd:w} are fulfilled, and that the constants $\alpha, L$ and $L_w$ can be chosen so that \eqref{e:balance} holds. Moreover, assume that the function $\lambda$ is Lipschitz continuous. Then the transition semigroup $\{P_t\}_{t\in\mathbb{R}_+}$ of the process $\Psi$, determined by \eqref{def:pdmp}, has a unique invariant probability measure, which belongs to $\mathcal{M}_1^V(X)$ with $V$ given by \eqref{def:V}.
\end{theorem}

\begin{proof}
First of all, according to Remark \ref{rem:usage}, we can apply \hbox{\cite[Theorem 4.1]{b:czapla_kubieniec}} to conclude that the transition operator $P$ of the chain $\Phi$, specified by \eqref{def:P_ifs}, possesses a~unique invariant probability measure~$\mu_*^{\Phi}$, which is a member of $\mathcal{M}_1^V(X)$. On the other hand, Lemma \ref{lem:f} guarantees that the kernel~$J$ fulfills Assumption \ref{cnd:f}. Consequently, it now follows from Theorem~\ref{thm:main} \hbox{(cf. also Corollary~\ref{cor:main})} that  $\mu_*^{\Psi}$ given by \eqref{e:m_psi} is a unique invariant probability measure of the semigroup $\{P_t\}_{t\in\mathbb{R}_+}$. Finally, Proposition \ref{prop:m1v}(i) yields that $\mu_*^{\Psi}\in\mathcal{M}_1^V(X)$, which ends the proof.
\end{proof}

\bibliography{ReferencesDatabase}

\section*{Statements and Declarations}

\textbf{Funding} The author declares that no funds, grants, or other support were received during the preparation of this manuscript.\bigskip

\noindent
\textbf{Competing Interests} The author has no relevant financial or non-financial interests to disclose.\bigskip

\noindent
\textbf{Conflict of Interest} The author declares that he has no conflict of interest.

\section*{Data Availability} 
Data sharing not applicable to this article as no datasets were generated or analysed during the current study.
\end{document}